\documentclass{amsart}

\usepackage[T1]{fontenc}
\usepackage[latin1]{inputenc}
\usepackage{url}
\usepackage[colorlinks=true]{hyperref}
\usepackage{enumitem}
\usepackage{dsfont}
\usepackage{graphicx}
\usepackage{amsthm}

\newcommand{\R}{\mathbb{R}}
\newcommand{\Exp}{\mathbb{E}}
\renewcommand{\Pr}{\mathbb{P}}
\newcommand{\dd}{\mathrm{d}}
\newcommand{\ind}[1]{\mathds{1}_{\{#1\}}}
\newcommand{\dto}{\downarrow}

\newcommand{\sgn}{\mathrm{sgn}}
\newcommand{\loc}{\mathrm{loc}}

\newcommand{\sk}{\vskip 3mm}
\newcommand{\ie}{{\em i.e.} }

\newtheorem{defi}{Definition}[section]
\newtheorem{lem}[defi]{Lemma}
\newtheorem{prop}[defi]{Proposition}

\newtheorem{cor}[defi]{Corollary}

\newtheoremstyle{myremark}{}{}{}{0pt}{\bfseries}{.}{ }{}
\theoremstyle{myremark}
\newtheorem{rk}[defi]{Remark}

\title{The small noise limit of order-based diffusion processes}

\author{Benjamin Jourdain}
\address{{\bf Jourdain, Benjamin}\newline
{\rm Université Paris-Est, Cermics (ENPC), F-77455 Marne-la-Vallée}}
\email{\href{mailto:jourdain@cermics.enpc.fr}{jourdain@cermics.enpc.fr}}

\author{Julien Reygner}
\address{{\bf Reygner, Julien}\newline
{\rm Sorbonne Universités, UPMC Univ Paris 06, UMR 7599, LPMA, F-75005 Paris\newline
Université Paris-Est, Cermics (ENPC), F-77455 Marne-la-Vallée}}
\email{\href{mailto:julien.reygner@upmc.fr}{julien.reygner@upmc.fr}}


\begin{document}

\begin{abstract}
  In this article, we introduce and study order-based diffusion processes. They are the solutions to multidimensional stochastic differential equations with constant diffusion matrix, proportional to the identity, and drift coefficient depending only on the ordering of the coordinates of the process. These processes describe the evolution of a system of Brownian particles moving on the real line with piecewise constant drifts, and are the natural generalization of the rank-based diffusion processes introduced in stochastic portfolio theory or in the probabilistic interpretation of nonlinear evolution equations. Owing to the discontinuity of the drift coefficient, the corresponding ordinary differential equations are ill-posed. Therefore, the small noise limit of order-based diffusion processes is not covered by the classical Freidlin-Wentzell theory. The description of this limit is the purpose of this article.
  
  We first give a complete analysis of the two-particle case. Despite its apparent simplicity, the small noise limit of such a system already exhibits various behaviours. In particular, depending on the drift coefficient, the particles can either stick into a cluster, the velocity of which is determined by elementary computations, or drift away from each other at constant velocity, in a random ordering. The persistence of randomness in the small noise limit is of the very same nature as in the pioneering works by Veretennikov (Mat. Zametki, 1983) and Bafico and Baldi (Stochastics, 1981) concerning the so-called Peano phenomenon.
  
  In the case of rank-based processes, we use a simple convexity argument to prove that the small noise limit is described by the sticky particle dynamics introduced by Brenier and Grenier (SIAM J. Numer. Anal., 1998), where particles travel at constant velocity between collisions, at which they stick together. In the general case of order-based processes, we give a sufficient condition on the drift for all the particles to aggregate into a single cluster, and compute the velocity of this cluster. Our argument consists in turning the study of the small noise limit into the study of the long time behaviour of a suitably rescaled process, and then exhibiting a Lyapunov functional for this rescaled process.
\end{abstract}

\maketitle

%%%%%%%%%%%%%%%%%%%%%%%%%%%%%%%%%%%%%%%%%%%%%%%%%%%%%%%%%%%%%%%%%%%%%%%%%%%%%%%%%%%%%%%%%%%%%%%%%%%%%%

\section{Introduction}

\subsection{Diffusions with small noise} The theory of ordinary differential equations (ODEs) with a regular drift coefficient and perturbed by a small stochastic noise was well developped by Freidlin and Wentzell~\cite{fw}. For a Lipschitz continuous function $b : \R^n \to \R^n$, they stated a large deviations principle for the laws of the solutions $X^{\epsilon}$ to the stochastic differential equations $\dd X^{\epsilon}(t) = b(X^{\epsilon}(t))\dd t + \sqrt{2\epsilon} \dd W(t)$, from which it can be easily deduced that $X^{\epsilon}$ converges to the unique solution to the ODE $\dot{x}=b(x)$ when $\epsilon$ vanishes. When the ODE $\dot{x}=b(x)$ is not well-posed, the behaviour of $X^{\epsilon}$ in the small noise limit is far less well understood. 

In one dimension of space, Veretennikov~\cite{ver83} and Bafico and Baldi~\cite{bafico} considered ODEs exhibiting a Peano phenomenon, \ie such that $b(0)=0$ and the ODE admits two continuous solutions $x^+$ and $x^-$ such that $x^+(0)=x^-(0)=0$, $x^+(t) > 0$ and $x^-(t) < 0$ for $t>0$. Other solutions are easily obtained for the ODE: as an example, for all $T>0$, the function $x^+_T$ defined by $x^+_T(t)=0$ if $t < T$ and $x^+_T(t) = x^+(t-T)$ if $t \geq T$ is also a continuous solution to the ODE. The solutions $x^+$ and $x^-$ are called {\em extremal} in the sense that they leave the origin instantaneously. For particular examples of such ODEs, it was proved in~\cite{ver83} and~\cite{bafico} that the small noise limit of the law of $X^{\epsilon}$ concentrates on the set of extremal solutions $\{x^+,x^-\}$ and the weights associated with each such solution was explicitely computed. In this case, large deviations principles were also proved by Herrmann~\cite{herrmann} and Gradinaru, Herrmann and Roynette~\cite{ghr}.

In higher dimensions of space, very few results are available. Buckdahn, Ouknine and Quincampoix~\cite{boq} proved that the limit points of the law of $X^{\epsilon}$ concentrate on the set of solutions to the ODE $\dot{x} = b(x)$ in the so-called Filippov generalized sense. However, an explicit description of this set is not easily provided in general. Let us also mention the work by Delarue, Flandoli and Vincenzi~\cite{dfv} in the specific setting of the Vlasov-Poisson equation on the real line for two electrostatic particles. For a particular choice of the electric field and of the initial conditions, they showed that the particles collapse in a finite time $T>0$, so that the ODE describing the Lagrangian dynamics of the two particles is singular at this time. After the singularity, the ODE exhibits a Peano-like phenomenon in the sense that it admits several extremal solutions, \ie leaving the singular point instantaneously. Similarly to the one-dimensional examples addressed in~\cite{ver83,bafico}, the trajectory obtained as the small noise limit of a stochastic perturbation is random among these extremal solutions.

\subsection{Order-based processes} In this article, we are interested in the small noise limit of the solution $X^{\epsilon}$ to the stochastic differential equation 
\begin{equation}\label{eq:sde}
  \forall t \geq 0, \qquad X^{\epsilon}(t) = x^0 + \int_{s=0}^t b(\Sigma X^{\epsilon}(s)) \dd s + \sqrt{2\epsilon} W(t),
\end{equation}
where $x^0 \in \R^n$, $b$ is a function from the symmetric group $S_n$ to $\R^n$, $W$ is a standard Brownian motion in $\R^n$ and, for $x=(x_1, \ldots, x_n) \in \R^n$, $\Sigma x$ is a permutation $\sigma \in S_n$ such that $x_{\sigma(1)} \leq \cdots \leq x_{\sigma(n)}$. A permutation $\sigma \in S_n$ shall sometimes be represented by the word $(\sigma(1) \cdots \sigma(n))$, especially for small values of $n$. As an example, the permutation $\sigma \in S_3$ defined by $\sigma(1)=2$, $\sigma(2)=1$ and $\sigma(3)=3$ is denoted by $(213)$.

On the set $O_n := \{x=(x_1, \ldots, x_n) \in \R^n : \exists i\not= j, x_i=x_j\}$ of vectors with non pairwise distinct coordinates, the permutation $\Sigma x$ is not uniquely defined. For the sake of precision, a convention to define $\Sigma x$ in this case is given below, although we prove in Proposition~\ref{prop:X} that the solution $X^{\epsilon}$ to~\eqref{eq:sde} does not depend on the definition of the quantity $b(\Sigma x)$ on $O_n$.

The solution $X^{\epsilon} = (X^{\epsilon}_1(t), \ldots, X^{\epsilon}_n(t))_{t \geq 0}$ to~\eqref{eq:sde} shall generically be called {\em order-based diffusion process}, as it describes the evolution of a system of $n$ particles moving on the real line with piecewise constant drift depending on their ordering. Note that, in such a system, the interactions can be nonlocal in the sense that a collision between two particles can modify the instantaneous drifts of all the particles in the system.

\sk
Section~\ref{s:n2} is dedicated to the complete description of the case $n=2$. Unsurprisingly, if the particles have distinct initial positions $x^0=(x^0_1,x^0_2)$, then in the small noise limit they first travel with constant velocity vector $b(\Sigma x^0)$. 

At a collision, or equivalently when the particles start from the same initial position, various behaviours are observed, depending on $b$. To describe these situations, a configuration $\sigma \in S_2$ is said to be {\em converging} if $b_{\sigma(1)}(\sigma) \geq b_{\sigma(2)}(\sigma)$, that is to say, the velocity of the leftmost particle is larger than the velocity of the rightmost particle, and {\em diverging} otherwise. If both configurations are converging, which writes 
\begin{equation*}
  b_1(12) \geq b_2(12), \qquad b_2(21) \geq b_1(21),
\end{equation*}
and shall be referred to as the {\em converging/converging case}, then, in the small noise limit, the particles stick together and form a cluster. The velocity of the cluster can be explicitely computed by elementary arguments. Except in some degenerate situations, it is deterministic and constant. If one of the configuration is converging while the other is diverging, which shall be referred to as the {\em converging/diverging case}, then, in the small noise limit, the particles drift away from each other with constant velocity vector $b(\sigma)$, where $\sigma$ is the diverging configuration. Finally, if both configurations are diverging, which writes
\begin{equation*}
  b_1(12) < b_2(12), \qquad b_2(21) < b_1(21),
\end{equation*}
and shall be referred to as the {\em diverging/diverging case}, then the particles drift away from each other with constant velocity vector $b(\sigma)$, where $\sigma$ is a random permutation in $S_2$ with an explicit distribution. 

The study of the two-particle case is made possible by the fact that most results actually stem from the study of the scalar process $Z^{\epsilon} := X^{\epsilon}_1 - X^{\epsilon}_2$. In particular, our result in the diverging/diverging case is similar to the situation of~\cite{ver83,bafico}, in the sense that the zero noise equation for $Z^{\epsilon}$ admits exactly two extremal solutions and exhibits a Peano phenomenon.

\sk
In higher dimensions, providing a general description of the small noise limit of $X^{\epsilon}$ seems to be a very challenging issue. As a first step, Sections~\ref{s:rank} and~\ref{s:clus} address two cases in which the function $b$ satisfies particular conditions. In Section~\ref{s:rank}, we assume that there exists a vector $b = (b_1, \ldots, b_n) \in \R^n$ such that, for all $\sigma \in S_n$, for all $i \in \{1, \ldots, n\}$, $b_{\sigma(i)}(\sigma)=b_i$. In other words, the instantaneous drift of the $i$-th particle does not depend on the whole ordering of $(X^{\epsilon}_1(t), \ldots, X^{\epsilon}_n(t))$, but only on the rank of $X^{\epsilon}_i(t)$ among $X^{\epsilon}_1(t), \ldots, X^{\epsilon}_n(t)$. In particular, the interactions are local in the sense that a collision between two particles does not affect the instantaneous drifts of the particles not involved in the collision. Such particle systems are generally called systems of {\em rank-based} interacting diffusions. They are of interest in the study of equity market models~\cite{fernholz, banner, pp, ferkar, ik, ipbkf, ips, fikp, fik, iks} or in the probabilistic interpretation of nonlinear evolution equations~\cite{bostal, jourdain:porous, jm, jourey, dsvz}.

A remarkable property of such systems is that the {\em reordered} particle system, defined as the process $Y^{\epsilon} = (Y^{\epsilon}_1(t), \ldots, Y^{\epsilon}_n(t))_{t \geq 0}$ such that, for all $t \geq 0$, $(Y^{\epsilon}_1(t), \ldots, Y^{\epsilon}_n(t))$ is the increasing reordering of $(X^{\epsilon}_1(t), \ldots, X^{\epsilon}_n(t))$, is a Brownian motion with constant drift vector $b$, normally reflected at the boundary of the polyhedron $D_n = \{(y_1, \ldots, y_n) \in \R^n : y_1 \leq \cdots \leq y_n\}$. By a simple convexity argument, we prove that the limit of $Y^{\epsilon}$ when $\epsilon$ vanishes is the deterministic process $\xi$ with the same drift $b$, normally reflected at the boundary of $D_n$.

The small noise limit $\xi$ turns out to coincide with the sticky particle dynamics introduced by Brenier and Grenier~\cite{bregre}, which describes the evolution of a system of particles with unit mass, travelling at constant velocity between collisions, and such that, at each collision, the colliding particles stick together and form a cluster, the velocity of which is determined by the global conservation of momentum. This provides an effective description of the small noise limit of $X^{\epsilon}$. 

\sk
An important fact in the rank-based case is that, whenever some particles form a cluster in the small noise limit, then for any partition of the cluster into a group of leftmost particles and a group of rightmost particles, the average velocity of the leftmost group is larger than the average velocity of the rightmost group. In Section~\ref{s:clus} we provide an extension of this stability condition to the general case of order-based diffusions. We prove that, when all the particles have the same initial position, this condition ensures that in the small noise limit, all the particles aggregate into a single cluster. However the condition is no longer necessary and we give a counterexample with $n=3$ particles.

To determine the motion of the cluster, we reinterpret the study of the small noise limit of $X^{\epsilon}$ as a problem of long time behaviour for the process $X^1$, thanks to an adequate change in the space and time scales. In the rank-based case, it is well known that $X^1$ does not have an equilibrium~\cite{pp,jm} as its projection along the direction $(1, \ldots, 1)$ is a Brownian motion with constant drift. However, under a stronger version of the stability condition, the orthogonal projection $Z^1$ of $X^1$ on the hyperplane $M_n = \{(z_1, \ldots, z_n) \in \R^n : z_1 + \cdots + z_n = 0\}$ admits a unique stationary distribution $\mu$. We extend both the strong stability condition and the existence and uniqueness result for $\mu$ to the order-based case, and thereby express the velocity of the cluster in terms of $\mu$.

\sk
In the conclusive Section~\ref{s:conclusion}, we state some conjectures as regards the general small noise limit of $X^{\epsilon}$, and we discuss the link between our results and the notion of {\em generalized flow} introduced by E and Vanden-Eijnden~\cite{eve}.

\subsection{Preliminary results and conventions} 

\subsubsection{Definition of $\Sigma$} For all $x \in \R^n$, we denote by $\bar{\Sigma}x$ the set of permutations $\sigma \in S_n$ such that $x_{\sigma(1)} \leq \cdots \leq x_{\sigma(n)}$. The set $\bar{\Sigma}x$ is nonempty, and it contains a unique element if and only if $x \not\in O_n$. The permutation $\Sigma x$ is defined as the lowest element of $\bar{\Sigma}x$ for the lexicographical order on the associated words. 

\subsubsection{Well-posedness of~\eqref{eq:sde}} Throughout this article, $x^0 \in \R^n$ refers to the initial positions of the particles, and a standard Brownian motion $W$ in $\R^n$ is defined on a given probability space $(\Omega, \mathcal{F}, \Pr_{x^0})$. The filtration generated by $W$ is denoted by $(\mathcal{F}_t)_{t \geq 0}$. The expectation under $\Pr_{x^0}$ is denoted by $\Exp_{x^0}$. 

\begin{prop}\label{prop:X}
  For all $\epsilon > 0$, for all $x^0 \in \R^n$, the stochastic differential equation~\eqref{eq:sde} admits a unique strong solution on the probability space $(\Omega, \mathcal{F}, \Pr_{x^0})$ provided with the filtration $(\mathcal{F}_t)_{t \geq 0}$. Besides, $\Pr_{x^0}$-almost surely,
  \begin{equation*}
    \forall t \geq 0, \qquad \int_{s=0}^t \ind{X^{\epsilon}(s) \in O_n} \dd s = 0.
  \end{equation*}
\end{prop}
\begin{proof}
  The strong existence and pathwise uniqueness follow from Veretennikov~\cite{veretennikov}, as the drift function $x \mapsto b(\Sigma x)$ is measurable and bounded, while the diffusion matrix is diagonal. The second part of the proposition is a consequence of the occupation time formula~\cite[p.~224]{revuz} applied to the semimartingales $X^{\epsilon}_i - X^{\epsilon}_j$, $i \not= j$.
\end{proof}

\subsubsection{Convergence of processes} Let $d \geq 1$. For all $T>0$, the space of continuous functions $C([0,T],\R^d)$ is endowed with the sup norm in time associated with the $L^1$ norm on $\R^d$. Let $A^{\epsilon} = (A_1^{\epsilon}(t), \ldots, A_d^{\epsilon}(t))_{t \geq 0}$ be a continuous process in $\R^d$ defined on the probability space $(\Omega, \mathcal{F}, \Pr_{x^0})$.

\begin{itemize}
  \item If $a = (a_1(t), \ldots, a_d(t))_{t \geq 0}$ is a continuous process in $\R^d$ defined on the probability space $(\Omega, \mathcal{F}, \Pr_{x^0})$, then for all $p \in [1,+\infty)$, $A^{\epsilon}$ is said to converge to $a$ in $L^p_{\loc}(\Pr_{x^0})$ if
  \begin{equation*}
    \forall T > 0, \qquad \lim_{\epsilon \dto 0} \Exp_{x^0}\left(\sup_{t \in [0,T]} \sum_{i=1}^d |A_i^{\epsilon}(t)-a_i(t)|^p\right) = 0.
  \end{equation*}

  \item If $a = (a_1(t), \ldots, a_d(t))_{t \geq 0}$ is a continuous process in $\R^d$ defined on some probability space $(\Omega', \mathcal{F}', \Pr')$, the process $A^{\epsilon}$ is said to converge in distribution to $a$ if, for all $T>0$, for all bounded continuous function $F : C([0,T],\R^d) \to \R$,
  \begin{equation*}
    \lim_{\epsilon \dto 0} \Exp_{x^0}(F(A^{\epsilon})) = \Exp'(F(a)),
  \end{equation*}
  where $\Exp'$ denotes the expectation under $\Pr'$, and, for the sake of brevity, the respective restrictions of $A^{\epsilon}$ and $a$ to $[0,T]$ are simply denoted by $A^{\epsilon}$ and $a$. 
\end{itemize}

Finally, the deterministic process $(t)_{t \geq 0}$ shall simply be denoted by $t$.

%%%%%%%%%%%%%%%%%%%%%%%%%%%%%%%%%%%%%%%%%%%%%%%%%%%%%%%%%%%%%%%%%%%%%%%%%%%%%%%%%%%%%%%%%%%%%%%%%%%%%%

\section{The two-particle case}\label{s:n2}

In this section, we assume that $n=2$. Then,~\eqref{eq:sde} rewrites
\begin{equation}\label{eq:sde:n2}
  X^{\epsilon}(t) = x^0 + b(12) \int_{s=0}^t \ind{X^{\epsilon}_1(s) \leq X^{\epsilon}_2(s)} \dd s + b(21) \int_{s=0}^t \ind{X^{\epsilon}_1(s) > X^{\epsilon}_2(s)} \dd s + \sqrt{2\epsilon} W(t).
\end{equation}

In the configuration $(12)$, that is to say whenever $X^{\epsilon}_1(t) \leq X^{\epsilon}_2(t)$, the instantaneous drift of the $i$-th particle is $b_i(12)$. Thus, in the small noise limit, the particles tend to get closer to each other if $b_1(12) \geq b_2(12)$, and to drift away from each other else. As a consequence, the configuration $(12)$ is said to be {\em converging} if $b^- := b_1(12)-b_2(12) \geq 0$ and {\em diverging} if $b^- < 0$. Similarly, the configuration $(21)$ is said to be {\em converging} if $b^+ := b_1(21)-b_2(21) \leq 0$ and {\em diverging} if $b^+ > 0$. The introduction of the quantities $b^-$ and $b^+$ is motivated by the fact that the reduced process $Z^{\epsilon} := X^{\epsilon}_1 - X^{\epsilon}_2$ satisfies the scalar stochastic differential equation
\begin{equation}\label{eq:Z}
  Z^{\epsilon}(t) = z^0 + \int_{s=0}^t \ell(Z^{\epsilon}(s)) \dd s + 2\sqrt{\epsilon} B(t),
\end{equation}
where $z^0:=x^0_1-x^0_2$, $\ell(z) := b^-\ind{z \leq 0} + b^+\ind{z > 0}$ and $B := (W_1-W_2)/\sqrt{2}$ is a standard Brownian motion in $\R$ defined on $(\Omega,\mathcal{F}, \Pr_x)$, adapted to the filtration $(\mathcal{F}_t)_{t \geq 0}$.

The description of the small noise limit of $X^{\epsilon}$ is exhaustively made in Subsection~\ref{ss:X}. Some proofs are postponed to Appendix~\ref{app:n2}. In Subsection~\ref{ss:Z}, the small noise limit of $Z^{\epsilon}$ is discussed. In the sequel, we use the terminology of~\cite{bafico} and call {\em extremal solution} to the zero noise version of~\eqref{eq:sde:n2} a continuous function $x = (x(t))_{t \geq 0}$ such that
\begin{equation*}
  \forall t \geq 0, \qquad x(t) = x^0 + \int_{s=0}^t b(\Sigma x(s)) \dd s,
\end{equation*}
and, for all $t > 0$, $x(t) \not\in O_2$.

\subsection{Small noise limit of the system of particles}\label{ss:X} To describe the small noise limit of $X^{\epsilon}$, we first address the case in which both particles have the same initial position, \ie $x^0 \in O_2$. The zero noise version of~\eqref{eq:sde:n2} rewrites
\begin{equation*}
  \forall t \geq 0, \qquad x(t) = x^0 + b(12) \int_{s=0}^t \ind{x_1(s) \leq x_2(s)} \dd s + b(21) \int_{s=0}^t \ind{x_1(s) > x_2(s)} \dd s.
\end{equation*}

In the diverging/diverging case $b^- < 0$, $b^+ > 0$, the equation above admits two extremal solutions $x^-$ and $x^+$ defined by $x^-(t) = x^0+b(12)t$ and $x^+ = x^0+b(21)t$. In the converging/diverging case $b^- \geq 0$, $b^+ > 0$, the only extremal solution is $x^+$, and symmetrically, in the case $b^- < 0$, $b^+ \leq 0$, the only extremal solution is $x^-$. In all these cases, the small noise limit of $X^{\epsilon}$ concentrates on the set of extremal solutions to the zero noise equation, similarly to the situations addressed in~\cite{ver83,bafico}.

\begin{prop}\label{prop:di}
  Assume that $x^0 \in O_2$, and recall that $x^-(t) = x^0+b(12)t$, $x^+(t) = x^0+b(21)t$.
  \begin{enumerate}[label=(\roman*), ref=\roman*]
    \item\label{case:didi} If $b^- < 0$, $b^+ > 0$, the process $X^{\epsilon}$ converges in distribution to $\rho x^+ + (1-\rho) x^-$ where $\rho$ is a Bernoulli variable with parameter $-b^-/(b^+-b^-)$.
    \item\label{case:condi} If $b^- \geq 0$, $b^+ > 0$, the process $X^{\epsilon}$ converges in $L^1_{\loc}(\Pr_{x^0})$ to $x^+$.
    \item\label{case:dicon} If $b^- < 0$, $b^+ \leq 0$, the process $X^{\epsilon}$ converges in $L^1_{\loc}(\Pr_{x^0})$ to $x^-$.
  \end{enumerate}
\end{prop}

In the converging/converging case $b^- \geq 0$, $b^+ \leq 0$, there is no extremal solution to the zero noise version of~\eqref{eq:sde:n2}. Informally, in both configurations the instantaneous drifts of each particle tend to bring the particles closer to each other. Therefore, in the small noise limit, the particles are expected to stick together and form a cluster; that is to say, the limit of the distribution of $X^{\epsilon}$ is expected to concentrate on $O_2$. The motion of the cluster is described in the following proposition.

\begin{prop}\label{prop:concon}
  Assume that $x^0 \in O_2$, and that $b^- \geq 0$, $b^+ \leq 0$.
  \begin{enumerate}[label=(\roman*), ref=\roman*, start=4]
    \item\label{case:concon} If $b^--b^+>0$, the process $X^{\epsilon}$ converges in $L^2_{\loc}(\Pr_{x^0})$ to $\rho x^- + (1-\rho)x^+$, where $\rho = -b^+/(b^--b^+)$ is the unique deterministic constant in $(0,1)$ such that, for all $t \geq 0$, $\rho x^-(t) + (1-\rho)x^+(t) \in O_2$.
    \item\label{case:conconl} If $b^-=b^+=0$, the process $X^{\epsilon}$ converges in $L^2_{\loc}(\Pr_{x^0})$ to $\rho x^- + (1-\rho)x^+$, where $\rho$ is the random process in $(0,1)$ defined by
    \begin{equation*}
      \forall t > 0, \qquad \rho(t) := \frac{1}{t} \int_{s=0}^t \ind{W_1(s) \leq W_2(s)} \dd s.
    \end{equation*}    
  \end{enumerate}
  Note that, in both cases, the small noise limit of $X^{\epsilon}$ takes its values in $O_2$.
\end{prop}

In other words, in case~\eqref{case:concon}, the cluster has a deterministic and constant velocity $v$ given by
\begin{equation*}
  v = \rho b_1(12) + (1-\rho) b_1(21) = \frac{b_2(21)b_1(12) - b_2(12)b_1(21)}{b_1(12)-b_2(12) - b_1(21)+b_2(21)}.
\end{equation*} 
In case~\eqref{case:conconl}, both particles have the same instantaneous drift in each of the two configurations, and the instantaneous drift of the cluster is a random linear interpolation of these drifts, with a coefficient $\rho(t)$ distributed according to the Arcsine law.

A common feature of Propositions~\ref{prop:di} and~\ref{prop:concon} is that, in all cases, the small noise limit of $X^{\epsilon}(t)$ is a linear interpolation of $x^-(t)$ and $x^+(t)$ with coefficients $\rho(t), 1-\rho(t) \in [0,1]$. Depending on the case at stake, $\rho(t)$ exhibits a wide range of various behaviours: in case~\eqref{case:didi}, it is random in $\{0,1\}$ and constant in time, in cases~\eqref{case:condi} and~\eqref{case:dicon} it is deterministic in $\{0,1\}$ and constant in time, in case~\eqref{case:concon} it is deterministic in $(0,1)$ and constant in time, and in case~\eqref{case:conconl} it is random in $(0,1)$ and nonconstant in time. 

In view of~\eqref{eq:sde:n2}, $\rho(t)$ appears as the natural small noise limit of the quantity $\zeta^{\epsilon}(t)/t$, where $\zeta^{\epsilon}$ denotes the occupation time of $X^{\epsilon}$ in the configuration $(12)$:
\begin{equation*}
  \forall t \geq 0, \qquad \zeta^{\epsilon}(t) := \int_{s=0}^t \ind{X^{\epsilon}_1(s) \leq X^{\epsilon}_2(s)} \dd s = \int_{s=0}^t \ind{Z^{\epsilon}(s) \leq 0} \dd s,
\end{equation*}
where we recall that $Z^{\epsilon} = X^{\epsilon}_1 - X^{\epsilon}_2$ solves~\eqref{eq:Z}. Indeed, Propositions~\ref{prop:di} and~\ref{prop:concon} easily stem from the following description of the small noise limit of the continuous process $\zeta^{\epsilon}$.

\begin{lem}\label{lem:zeta}
  Assume that $x^0 \in O_2$.
  \begin{enumerate}[label=(\roman*), ref=\roman*]
    \item\label{case:didi:zeta} If $b^- < 0$, $b^+ > 0$, then $\zeta^{\epsilon}$ converges in distribution to the process $\rho t$, where $\rho$ is a Bernoulli variable with parameter $-b^-/(b^+-b^-)$.
    \item\label{case:condi:zeta} If $b^- \geq 0$, $b^+ > 0$, then $\zeta^{\epsilon}$ converges in $L^1_{\loc}(\Pr_{x^0})$ to $0$.
    \item\label{case:dicon:zeta} If $b^- < 0$, $b^+ \leq 0$, then $\zeta^{\epsilon}$ converges in $L^1_{\loc}(\Pr_{x^0})$ to $t$.
    \item\label{case:concon:zeta} If $b^- \geq 0$, $b^+ \leq 0$ and $b^--b^+>0$, then $\zeta^{\epsilon}$ converges in $L^2_{\loc}(\Pr_{x^0})$ to $\rho t$, where $\rho = -b^+/(b^--b^+)$.
    \item\label{case:conconl:zeta} If $b^-=b^+=0$, then for all $t \geq 0$,
    \begin{equation*}
      \zeta^{\epsilon}(t) = \int_{s=0}^t \ind{W_1(s) \leq W_2(s)}\dd s.
    \end{equation*}
  \end{enumerate}
\end{lem}
\begin{proof}
  The proofs of cases~\eqref{case:didi:zeta}, \eqref{case:condi:zeta} and~\eqref{case:dicon:zeta} are given in Appendix~\ref{app:n2}. The proof of case~\eqref{case:concon:zeta} is an elementary computation and is given in Subsection~\ref{ss:Z} below. In case~\eqref{case:conconl:zeta}, there is nothing to prove.
\end{proof}

\begin{rk}\label{rk:blum}
  In cases~\eqref{case:condi:zeta}, \eqref{case:dicon:zeta}, and~\eqref{case:concon:zeta} above, the convergence is stated either in $L^1_{\loc}(\Pr_{x^0})$ or in $L^2_{\loc}(\Pr_{x^0})$ as these modes of convergence appear most naturally in the proof. However, all our arguments can easily be extended to show that all the convergences hold in $L^p_{\loc}(\Pr_{x^0})$, for all $p \in [1,+\infty)$. As a consequence, all the convergences in Proposition~\ref{prop:di} and~\ref{prop:concon}, except in case~\eqref{case:didi}, actually hold in $L^p_{\loc}(\Pr_{x^0})$, for all $p \in [1,+\infty)$. 
  
  On the contrary, the convergence in the diverging/diverging case~\eqref{case:didi:zeta} cannot hold in probability. Indeed, let us assume by contradiction that there exists $T>0$ such that the convergence in case~\eqref{case:didi:zeta} of Lemma~\ref{lem:zeta} holds in probability in $C([0,T],\R)$. Then, for all $t \in [0,T]$, $\zeta^{\epsilon}(t)$ converges in probability to $\rho t$. Let us fix $t \in (0,T]$. By Proposition~\ref{prop:X}, for all $\epsilon > 0$, the random variable $\zeta^{\epsilon}(t)$ is measurable with respect to the $\sigma$-field $\mathcal{F}_t$ generated by $(W(s))_{s \in [0,t]}$. Thus, we deduce that the random variable $\rho$ is measurable with respect to $\mathcal{F}_t$. As a consequence, $\rho$ is measurable with respect to $\mathcal{F}_{0^+} := \cap_{t > 0} \mathcal{F}_t$, which is contradictory with the Blumenthal zero-one law for the Brownian motion $W$.
  
  We finally mention that in cases~\eqref{case:didi:zeta}, \eqref{case:condi:zeta}, \eqref{case:dicon:zeta} and~\eqref{case:concon:zeta}, the small noise limit of the process $\zeta^{\epsilon}$ is a Markov process, which is not the case for the process $\zeta^{\epsilon}$ itself.
\end{rk}

\sk
Let us now address the case $x^0 \not\in O_2$ of particles with distinct initial positions. Let $\sigma = \Sigma x^0$. If $b_{\sigma(1)}(\sigma) \leq b_{\sigma(2)}(\sigma)$, a pair of particles travelling at constant velocity vector $b(\sigma)$ with initial positions $x^0$ never collides, and the natural small noise limit of $X^{\epsilon}$ is given by $x(t) = x^0 + b(\sigma) t$, for all $t \geq 0$.

If $b_{\sigma(1)}(\sigma) > b_{\sigma(2)}(\sigma)$, a pair of particles travelling at constant velocity vector $b(\sigma)$ with initial positions $x^0$ collides at time $t^*(x^0) := -(x^0_1 - x^0_2)/(b_1(\sigma)-b_2(\sigma)) \in (0,+\infty)$. The natural small noise limit of $X^{\epsilon}$ is now described by $x(t) = x^0 + b(\sigma) t$ for $t < t^*(x^0)$, and for $t \geq t^*(x^0)$, $x(t)$ is the small noise limit of $X'^{\epsilon}(t-t^*(x^0))$, where $X'^{\epsilon}$ is a copy of $X^{\epsilon}$ started at $x^0 + b(\sigma) t^*(x^0) \in O_2$. In that case, at least the configuration $\sigma$ is converging, therefore there is neither random selection of a trajectory as in case~\eqref{case:didi}, nor random and nonconstant velocity of the cluster as in case~\eqref{case:conconl}.

These statements are straightforward consequences of the description of the small noise limit of the process $Z^{\epsilon}$ with $z^0 \not= 0$ carried out in Corollary~\ref{cor:Z} below.

\subsection{The reduced process}\label{ss:Z} By Veretennikov~\cite{veretennikov}, strong existence and pathwise uniqueness hold for~\eqref{eq:Z}; therefore, for all $\epsilon > 0$, $Z^{\epsilon}$ is adaptated to the filtration generated by the Brownian motion $B$. As a consequence, the probability of a measurable event $A$ with respect to the $\sigma$-field generated by $(B(s))_{s \in [0,t]}$ for some $t \geq 0$ shall be abusively denoted by $\Pr_{z^0}(A)$ instead of $\Pr_{x^0}(A)$.

\sk
To describe the small noise limit of $Z^{\epsilon}$, we define $z^-(t) = b^-t$ and $z^+(t) = b^+t$. Let us begin with the case $z^0=0$, which corresponds to $x^0 \in O_2$.

\begin{prop}\label{prop:Z}
  Assume that $z^0=0$. Then,
  \begin{enumerate}[label=(\roman*), ref=\roman*]
    \item\label{case:didi:z} if $b^- < 0$ and $b^+ > 0$, then $Z^{\epsilon}$ converges in distribution to $\rho z^- + (1-\rho) z^+$, where $\rho$ is a Bernoulli variable of parameter $-b^-/(b^+-b^-)$;
    \item\label{case:condi:z} if $b^- \geq 0$ and $b^+ > 0$, then $Z^{\epsilon}$ converges to $z^+$ in $L^1_{\loc}(\Pr_0)$;
    \item\label{case:dicon:z} if $b^- < 0$ and $b^+ \leq 0$, then $Z^{\epsilon}$ converges to $z^-$ in $L^1_{\loc}(\Pr_0)$;
    \item\label{case:concon:z} if $b^- \geq 0$ and $b^+ \leq 0$, then $Z^{\epsilon}$ converges to $0$ in $L^2_{\loc}(\Pr_0)$; more precisely,
    \begin{equation*}
      \forall T>0, \qquad \Exp_0\left(\sup_{t \in [0,T]} |Z^{\epsilon}(t)|^2\right) \leq (8\sqrt{2}+4)\epsilon T.
    \end{equation*}
  \end{enumerate}
\end{prop}
\begin{proof}
  Since $Z^{\epsilon}(t) = b^- \zeta^{\epsilon}(t) + b^+ (t - \zeta^{\epsilon}(t)) + 2\sqrt{\epsilon}B(t)$, cases~\eqref{case:didi:z}, \eqref{case:condi:z} and~\eqref{case:dicon:z} are straightforward consequences of the corresponding statements in Lemma~\ref{lem:zeta}, the proofs of which are given in Appendix~\ref{app:n2}. 
  
  We now give a direct proof of case~\eqref{case:concon:z}. By the Itô formula, for all $t \geq 0$,
  \begin{equation*}
    |Z^{\epsilon}(t)|^2 = 2\int_{s=0}^t Z^{\epsilon}(s)\ell(Z^{\epsilon}(s))\dd s + 4\sqrt{\epsilon}\int_{s=0}^t Z^{\epsilon}(s) \dd B(s) + 4\epsilon t.
  \end{equation*}
  If $b^+ \leq 0$ and $b^- \geq 0$, then for all $z \in \R$ one has $z\ell(z) \leq 0$, therefore
  \begin{equation*}
    |Z^{\epsilon}(t)|^2 \leq 4\sqrt{\epsilon}\int_{s=0}^t Z^{\epsilon}(s) \dd B(s) + 4\epsilon t.
  \end{equation*}
  
  For all $t \geq 0$, let us define
  \begin{equation*}
    M^{\epsilon}(t) = \int_{s=0}^t Z^{\epsilon}(s) \dd B(s);
  \end{equation*}
  and for all $L>0$, let $\tau_L := \inf\{t \geq 0 : |Z^{\epsilon}(t)| \geq L\}$. The process $(M^{\epsilon}(t \wedge \tau_L))_{t \geq 0}$ is a martingale, therefore, for all $t \geq 0$, $\Exp_0(|Z^{\epsilon}(t \wedge \tau_L)|^2) \leq 4\epsilon \Exp_0(t \wedge \tau_L) \leq 4\epsilon t$, and by the Fatou lemma, $\Exp_0(|Z^{\epsilon}(t)|^2) \leq 4\epsilon t$. As a consequence, $(M^{\epsilon}(t))_{t \geq 0}$ is a martingale. For all $T>0$,
  \begin{equation*}
    \begin{aligned}
      \Exp_0\left(\sup_{t \in [0,T]} |Z^{\epsilon}(t)|^2\right) & \leq 4\sqrt{\epsilon}\Exp_0\left(\sup_{t \in [0,T]}  M^{\epsilon}(t)\right) + 4\epsilon T\\
      & \leq 4\sqrt{\epsilon}\sqrt{\Exp_0\left(\sup_{t \in [0,T]}  M^{\epsilon}(t)^2\right)} + 4\epsilon T\\
      & \leq 8\sqrt{\epsilon}\sqrt{\Exp_0\left(M^{\epsilon}(T)^2\right)} + 4\epsilon T\\
      & = 8\sqrt{\epsilon}\sqrt{\Exp_0\left(\int_0^T Z^{\epsilon}(s)^2\dd s\right)} + 4\epsilon T\\
      & \leq 8\sqrt{\epsilon}\sqrt{\int_0^T 4 \epsilon s \dd s} + 4\epsilon T = (8\sqrt{2}+4) \epsilon T,
    \end{aligned}
  \end{equation*}
  where we have used the Cauchy-Schwarz inequality at the second line, the Doob inequality at the third line and the Itô isometry at the fourth line. This completes the proof of case~\eqref{case:concon:z}.
\end{proof}

In case~\eqref{case:concon:z} of Lemma~\ref{lem:zeta}, the computation of the small noise limit of $\zeta^{\epsilon}$ is straightforward.

\begin{proof}[Proof of case~\eqref{case:concon:zeta} in Lemma~\ref{lem:zeta}]
  Let $T>0$. By case~\eqref{case:concon:z} in Proposition~\ref{prop:Z}, if $x^0 \in O_2$ and $b^- \geq 0$, $b^+ \leq 0$, then $\lim_{\epsilon \dto 0} b^- \zeta^{\epsilon} + b^+ (t - \zeta^{\epsilon}) = 0$ in $L^2_{\loc}(\Pr_{x^0})$. If $b^--b^+>0$ in addition, then this relation yields $\lim_{\epsilon \dto 0} \zeta^{\epsilon} = \rho t$ in $L^2_{\loc}(\Pr_{x^0})$, with $\rho = -b^+/(b^--b^+)$.
\end{proof}

We now describe the small noise limit of $Z^{\epsilon}$ in the case $z^0 \not= 0$. Due to the same reasons as in Remark~\ref{rk:blum}, all the convergences below are stated in $L^1_{\loc}(\Pr_0)$ but can easily be extended to $L^p_{\loc}(\Pr_0)$ for all $p \in [1,+\infty)$. The proof of Corollary~\ref{cor:Z} is postponed to Appendix~\ref{app:n2}.

\begin{cor}\label{cor:Z}
  Assume that $z^0>0$. Let us define $t^* = +\infty$ if $b^+ \geq 0$, and $t^* := z^0/(-b^+)$ if $b^+ < 0$. Then $Z^{\epsilon}$ converges in $L^1_{\loc}(\Pr_{z^0})$ to the process $z^{\dto}$ defined by:
  \begin{equation*}
    \forall t \geq 0, \qquad z^{\dto}(t) := \left\{\begin{aligned}
      & z^0 + b^+t & \text{if $t < t^*$},\\
      & 0 & \text{if $t \geq t^*$ and $b^- \geq 0$},\\
      & b^-(t-t^*) & \text{if $t \geq t^*$ and $b^- < 0$}.
    \end{aligned}\right.
  \end{equation*}
  A symmetric statement holds if $z^0<0$.
\end{cor}

\begin{rk} 
  For a given continuous and bounded function $u_0$ on $\R$, the function $u^{\epsilon}$ defined by
  \begin{equation*}
    \forall (t,z) \in [0,+\infty)\times\R, \qquad u^{\epsilon}(t,z) := \Exp_z(u_0(Z^{\epsilon}(t)))
  \end{equation*}
  is continuous on $[0,+\infty)\times\R$ owing to the Girsanov theorem and the boundedness of $\ell$. Following~\cite[Chapter~II]{fleson}, it is a viscosity solution to the parabolic Cauchy problem
  \begin{equation*}
    \left\{\begin{aligned}
      & \partial_t u^{\epsilon} - \ell(z) \partial_z u^{\epsilon} = 2 \epsilon \partial_{zz} u^{\epsilon},\\
      & u^{\epsilon}(0,\cdot) = u_0(\cdot).
    \end{aligned}\right.
  \end{equation*}
  Attanasio and Flandoli~\cite{flandoli} addressed the limit of $u^{\epsilon}$ when $\epsilon$ vanishes, for a particular function $\ell$ such that the corresponding hyperbolic Cauchy problem
  \begin{equation*}
    \left\{\begin{aligned}
      & \partial_t u^{\epsilon} - \ell(z) \partial_z u^{\epsilon} = 0,\\
      & u^{\epsilon}(0,\cdot) = u_0(\cdot),
    \end{aligned}\right.
  \end{equation*}
  admits several solutions. In the diverging/diverging case $b^+>0$, $b^-<0$, we recover their result of~\cite[Theorem~2.4]{flandoli} as $u^{\epsilon}$ converges pointwise to the function $u$ defined by 
  \begin{equation*}
    u(t,z) = \left\{\begin{aligned}
      & u_0(z+b^+t) & \text{if $z>0$},\\
      & u_0(z+b^-t) & \text{if $z<0$},\\
      & \frac{b^+}{b^+-b^-}u_0(b^+t) + \frac{-b^-}{b^+-b^-}u_0(b^-t) & \text{if $z=0$}.
    \end{aligned}\right.
  \end{equation*}
  Note that, in general, $u$ is discontinuous on the half line $z=0$.
  
  \begin{figure}[ht]
    \begin{center}
      \includegraphics{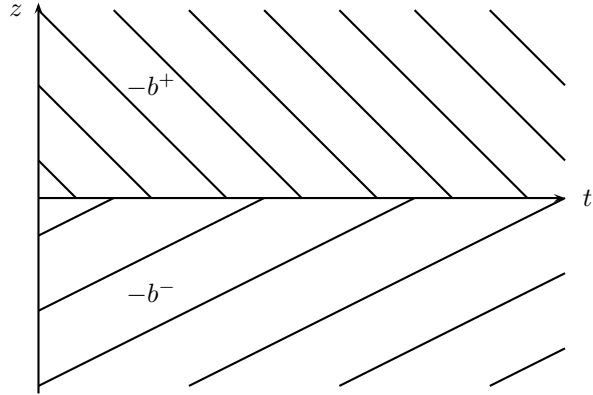}
      \caption{The characteristics of the conservation law in the diverging/diverging case. On the half line $z=0$, the value of $u$ is a linear interpolation of the values given by the upward characteristic and the downward characteristic.}
    \end{center}
  \end{figure}

  In the converging/converging case $b^+ \leq 0$, $b^- \geq 0$, $u^{\epsilon}$ converges pointwise to the function $u$ defined by
  \begin{equation*}
    u(t,z) = \left\{\begin{aligned}
      & u_0(z+b^+t) & \text{if $z > -b^+t$},\\
      & u_0(z+b^-t) & \text{if $z < -b^-t$},\\
      & u_0(0) & \text{if $-b^- t \leq z \leq -b^+ t$}.
    \end{aligned}\right.
  \end{equation*}
  Note that $u$ is continuous on $[0,+\infty)\times\R$, and constant on the cone $-b^- t \leq z \leq -b^+ t$.
  \begin{figure}[ht]
    \begin{center}
      \includegraphics{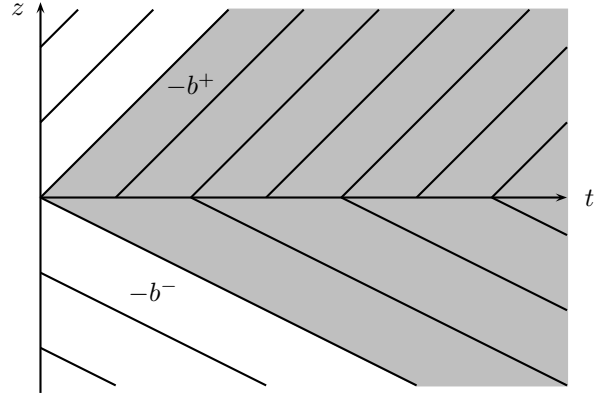}
      \caption{The characteristics in the converging/converging case. In the gray area, the value of $u$ is constant.}
    \end{center}
  \end{figure}
\end{rk}

%%%%%%%%%%%%%%%%%%%%%%%%%%%%%%%%%%%%%%%%%%%%%%%%%%%%%%%%%%%%%%%%%%%%%%%%%%%%%%%%%%%%%%%%%%%%%%%%%%%%%%

\section{The rank-based case}\label{s:rank}

In this section, we assume that there exists a vector $b = (b_1, \ldots, b_n) \in \R^n$ such that, for all $\sigma \in S_n$, for all $i \in \{1, \ldots, n\}$, $b_{\sigma(i)}(\sigma) = b_i$. In other words, the instantaneous drift of the $i$-th particle at time $t$ only depends on the rank of $X^{\epsilon}_i(t)$ among $X^{\epsilon}_1(t), \ldots, X^{\epsilon}_n(t)$. We recall in Subsection~\ref{ss:rps} that, in this case, the increasing reordering of the particle system is a Brownian motion with constant drift, normally reflected at the boundary of the polyhedron $D_n := \{(y_1, \ldots, y_n) \in \R^n : y_1 \leq \cdots \leq y_n\}$. Its small noise limit is obtained through a simple convexity argument, and identified as the sticky particle dynamics in Subsection~\ref{ss:sticky}. The description of the small noise limit of the original particle system is then derived in Subsection~\ref{ss:rbX}.

\subsection{The reordered particle system}\label{ss:rps} For all $t \geq 0$, let 
\begin{equation*}
  (Y^{\epsilon}_1(t), \ldots, Y^{\epsilon}_n(t)) \in D_n
\end{equation*}
refer to the increasing reordering of 
\begin{equation*}
  (X^{\epsilon}_1(t), \ldots, X^{\epsilon}_n(t)) \in \R^n,
\end{equation*}
\ie $Y^{\epsilon}_i(t) = X^{\epsilon}_{\sigma(i)}(t)$ with $\sigma = \Sigma X^{\epsilon}(t)$. The increasing reordering of the initial positions $x^0$ is denoted by $y^0$. The process $Y^{\epsilon} = (Y^{\epsilon}_1(t), \ldots, Y^{\epsilon}_n(t))_{t \geq 0}$ shall be referred to as the {\em reordered particle system}. It is continuous and takes its values in the polyhedron $D_n$. The following lemma is an easy adaptation of~\cite[Lemma~2.1, p.~91]{jourdain:porous}.

\begin{lem}\label{lem:spr}
  For all $\epsilon > 0$, there exists a standard Brownian motion 
  \begin{equation*}
    \beta^{\epsilon} = (\beta^{\epsilon}_1(t), \ldots, \beta^{\epsilon}_n(t))_{t \geq 0}
  \end{equation*}
  in $\R^n$, defined on $(\Omega,\mathcal{F},\Pr_{x^0})$, such that
  \begin{equation}\label{eq:rsde}
    \forall t \geq 0, \qquad Y^{\epsilon}(t) = y^0 + b t + \sqrt{2\epsilon}\beta^{\epsilon}(t) + K^{\epsilon}(t),
  \end{equation}
  where the continuous process $K^{\epsilon} = (K^{\epsilon}_1(t), \ldots, K^{\epsilon}_n(t))_{t \geq 0}$ in $\R^n$ is {\em associated} with $Y^{\epsilon}$ in $D_n$ in the sense of Tanaka~\cite[p.~165]{tanaka}. In other words, $Y^{\epsilon}$ is a Brownian motion with constant drift vector $b$ and constant diffusion matrix $2\epsilon I_n$, normally reflected at the boundary of the polyhedron $D_n$; where $I_n$ refers to the identity matrix.
\end{lem}

By Tanaka~\cite[Theorem~2.1, p.~170]{tanaka}, there exists a unique solution 
\begin{equation*}
  \xi = (\xi_1(t), \ldots, \xi_n(t))_{t \geq 0}
\end{equation*}
to the zero noise version of the reflected equation~\eqref{eq:rsde} given by
\begin{equation}\label{eq:xi}
  \forall t \geq 0, \qquad \xi(t) = y^0 + bt + \kappa(t),
\end{equation}
where $\kappa$ is associated with $\xi$ in $D_n$. An explicit description of $\xi$ as the sticky particle dynamics started at $y$ with initial velocity vector $b$ is provided in Subsection~\ref{ss:sticky} below. 

\begin{prop}\label{prop:contredsr}
  For all $T>0$,
  \begin{equation*}
    \Exp_{x^0}\left(\sup_{t \in [0,T]} \sum_{i=1}^n \left| Y^{\epsilon}_i(t) - \xi_i(t)\right|^2\right) \leq (4\sqrt{2n} + 2n)\epsilon T.
  \end{equation*}
\end{prop}
\begin{proof}
  By the Itô formula,
  \begin{equation*}
    \begin{aligned}
      \forall t \geq 0, \qquad \sum_{i=1}^n |Y^{\epsilon}_i(t)-\xi_i(t)|^2 & = 2\sum_{i=1}^n \int_{s=0}^t (Y^{\epsilon}_i(s)-\xi_i(s))\dd K^{\epsilon}_i(s)\\
      & \quad + 2\sum_{i=1}^n \int_{s=0}^t (\xi_i(s)-Y^{\epsilon}_i(s))\dd \kappa_i(s)\\
      & \quad + 2\sqrt{2\epsilon} M^{\epsilon}(t) + 2n\epsilon t,
    \end{aligned}
  \end{equation*}
  where
  \begin{equation*}
    \forall t \geq 0, \qquad M^{\epsilon}(t) := \sum_{i=1}^n \int_{s=0}^t (Y^{\epsilon}_i(s)-\xi_i(s))\dd \beta^{\epsilon}_i(s).
  \end{equation*}
  
  Let $|K^{\epsilon}|(t)$ refer to the total variation of $K^{\epsilon}$ on $[0,t]$. Then, by the definition of $K^{\epsilon}$ (see~\cite[p.~165]{tanaka}), $\dd |K^{\epsilon}|(t)$-almost everywhere, $Y^{\epsilon}(t) \in \partial D_n$ and the unit vector $k^{\epsilon}(t) = (k^{\epsilon}_1(t), \ldots, k^{\epsilon}_n(t))$ defined by $\dd K^{\epsilon}_i(t) = k^{\epsilon}_i(t)\dd |K^{\epsilon}|(t)$ belongs to the cone of inward normal vectors to $D_n$ at $Y^{\epsilon}(t)$. Since $\xi(t) \in D_n$ and the set $D_n$ is convex, this yields
  \begin{equation*}
    \sum_{i=1}^n \int_{s=0}^t (Y^{\epsilon}_i(s)-\xi_i(s))\dd K^{\epsilon}_i(s) = \int_{s=0}^t \sum_{i=1}^n (Y^{\epsilon}_i(s)-\xi_i(s))k^{\epsilon}_i(s) \dd |K^{\epsilon}|(s) \leq 0,
  \end{equation*}
  and by the same arguments,
  \begin{equation*}
    \sum_{i=1}^n \int_{s=0}^t (\xi_i(s)-Y^{\epsilon}_i(s))\dd \kappa_i(s) \leq 0,
  \end{equation*}
  so that $\sum_{i=1}^n |Y^{\epsilon}_i(t)-\xi_i(t)|^2 \leq 2\sqrt{2\epsilon} M^{\epsilon}(t) + 2n\epsilon t$. The result now follows from the same localization procedure as in the proof of Proposition~\ref{prop:Z}, case~\eqref{case:concon:z}.
\end{proof}

\subsection{The sticky particle dynamics}\label{ss:sticky} Following Brenier and Grenier~\cite{bregre}, the {\em sticky particle dynamics} started at $y^0 \in D_n$ with initial velocity vector $b \in \R^n$ is defined as the continuous process 
\begin{equation*}
  \xi = (\xi_1(t), \ldots, \xi_n(t))_{t \geq 0}
\end{equation*}
in $D_n$ satisfying the following conditions.
\begin{itemize}
  \item For all $i \in \{1, \ldots, n\}$, the $i$-th particle has initial position $\xi_i(0)=y^0_i$, initial velocity $b_i$ and unit mass.
  \item A particle travels with constant velocity until it collides with another particle. Then both particles stick together and form a cluster traveling at constant velocity given by the average velocity of the two colliding particles.
  \item More generally, when two clusters collide, they form a single cluster, the velocity of which is determined by the conservation of global momentum.
\end{itemize}

Certainly, particles with the same initial position can collide instantaneously and form one or several clusters, each cluster being composed by particles with consecutives indices. The determination of these instantaneous clusters is made explicit in ~\cite[Remarque~1, p.~235]{jourdain:sticky}.

Since the particles stick together after each collision, there is only a finite number $M \geq 0$ of collisions. Let us denote by $0 = t^0 < t^1 < \cdots < t^M < t^{M+1} = +\infty$ the instants of collisions. For all $m \in \{0, \ldots, M\}$, we define the equivalence relation $\sim_m$ by $i \sim_m j$ if the $i$-th particle and the $j$-th particle travel in the same cluster on $[t^m, t^{m+1})$. Note that if $i \sim_m j$, then $i \sim_{m'} j$ for all $m' \geq m$. For all $m \in \{0, \ldots, M\}$, we denote by $v^m_i$ the velocity of the $i$-th particle after the $m$-th collision. As a consequence, for all $t \in [t^m,t^{m+1})$,
\begin{equation*}
  \forall i \in \{1, \ldots, n\}, \qquad \xi_i(t) = \xi_i(t^m) + v^m_i(t-t^m),
\end{equation*}
and 
\begin{equation*}
  v^m_i = \frac{1}{i_2-i_1+1} \sum_{j=i_1}^{i_2} b_j,
\end{equation*}
where $\{i_1, \ldots, i_2\}$ is the set of the consecutive indices $j$ such that $j \sim_m i$. The clusters are characterized by the following {\em stability condition} due to Brenier and Grenier~\cite[Lemma~2.2, p.~2322]{bregre}.

\begin{lem}\label{lem:stab}
  For all $t \in [t^m,t^{m+1})$, for all $i \in \{1, \ldots, n\}$, let $i_1, \ldots, i_2$ refer to the set of consecutive indices $j$ such that $j \sim_m i$. Then, either $i_1=i_2$ or
  \begin{equation*}
    \forall i' \in \{i_1, \ldots, i_2-1\}, \qquad \frac{1}{i'-i_1+1} \sum_{j=i_1}^{i'} b_j \geq \frac{1}{i_2-i'} \sum_{j=i'+1}^{i_2} b_j. 
  \end{equation*}
\end{lem}

The fact that $\xi$ describes the small noise limit of the reordered particle system $Y^{\epsilon}$ introduced in Subsection~\ref{ss:rps} is a consequence of Proposition~\ref{prop:contredsr} combined with the following lemma.

\begin{lem}\label{lem:xi}
  The process $\xi$ satisfies the reflected equation~\eqref{eq:xi} in $D_n$.
\end{lem}
\begin{proof}
  The proof is constructive, namely we build a process $\kappa$ associated with $\xi$ in $D_n$ such that, for all $t \geq 0$, $\xi(t)=y^0+bt+\kappa(t)$. Following~\cite[Remark~2.3, p.~91]{jourdain:porous}, $\kappa : [0,+\infty) \to \R^n$ is associated with $\xi$ in $D_n$ if and only if:
  \begin{itemize}
    \item[(i)] $\kappa$ is continuous, with bounded variation $|\kappa|=|\kappa_1| + \cdots + |\kappa_n|$ and $\kappa(0)=0$;
    \item[(ii)] there exist functions $\gamma_1, \ldots, \gamma_{n+1} : [0,+\infty) \to \R$ such that, for all $i \in \{1, \ldots, n\}$, $\dd \kappa_i(t) = (\gamma_i(t)-\gamma_{i+1}(t))\dd |\kappa|(t)$; and, $\dd|\kappa|(t)$-almost everywhere,
    \begin{equation*}
      \begin{aligned}
        & \gamma_1(t)=\gamma_{n+1}(t)=0,\\
        & \forall i \in \{2, \ldots, n\}, \quad \gamma_i(t) \geq 0, \quad \gamma_i(t)(\xi_i(t)-\xi_{i-1}(t))=0.
      \end{aligned}
    \end{equation*}
  \end{itemize}
  
  Let $\kappa(0)=0$ and let us define $\kappa_i(t) = \kappa_i(t^m) + (t-t^m)(v^m_i-b_i)$ for all $t \in [t^m, t^{m+1})$. Then one easily checks that~\eqref{eq:xi} holds. Besides, $\kappa$ is absolutely continuous with respect to the Lebesgue measure on $[0,+\infty)$, and its total variation $|\kappa|$ admits the Radon-Nikodym derivative $\ell_m := \sum_{i=1}^n |v^m_i-b_i|$ on $[t^m, t^{m+1})$. As a consequence, $\kappa$ satisfies (i). 
  
  It remains to prove that $\kappa$ satisfies (ii). For all $m \in \{0, \ldots, M\}$, for all $t \in [t^m, t^{m+1})$, we define $\gamma_1(t)=\gamma_{n+1}(t)=0$ and:
  \begin{itemize}
    \item if $\ell_m=0$, for all $i \in \{2, \ldots, n\}$, $\gamma_i(t)=0$;
    \item if $\ell_m>0$, for all $i \in \{2, \ldots, n\}$,
    \begin{equation*}
      \gamma_i(t) = \frac{1}{\ell_m} \sum_{j=i_1}^{i-1}(b_j-v_i^m),
    \end{equation*}
    where $i_1, \ldots, i_2$ is the set of the consecutive indices $j$ such that $j \sim_m i$, and we take the convention that a sum over an empty set of indices is null.
  \end{itemize}
  Note that, in the latter case, $\gamma_{i_1}(t)=\gamma_{i_2+1}(t)=0$. This immediately yields $\dd \kappa_i(t) = (\gamma_i(t)-\gamma_{i+1}(t))\dd|\kappa|(t)$ as well as $\gamma_i(t)(\xi_i(t)-\xi_{i-1}(t))=0$. It remains to prove that $\gamma_i(t) \geq 0$. If $\gamma_i(t)=0$ this is trivial. Else, by the construction above, the $i$-th particle belongs to the cluster composed by the $i_1\text{-th}, \ldots, i_2\text{-th}$ particles, and $i_1 < i \leq i_2$. By Lemma~\ref{lem:stab} applied with $i'=i-1$,
  \begin{equation*}
    \frac{1}{i-i_1}\sum_{j=i_1}^{i-1} b_j \geq \frac{1}{i_2-i+1}\sum_{j=i}^{i_2} b_j.
  \end{equation*} 
  As a consequence,
  \begin{equation*}
    \begin{aligned}
      \gamma_i(t) = \frac{1}{\ell_m} \sum_{j=i_1}^{i-1}(b_j-v_i^m) & = \frac{1}{\ell_m}\left(\sum_{j=i_1}^{i-1} b_j - \frac{i-i_1}{i_2-i_1+1} \sum_{j=i_1}^{i_2} b_j\right)\\
      & = \frac{1}{\ell_m}\left(\frac{i_2-i+1}{i_2-i_1+1}\sum_{j=i_1}^{i-1} b_j - \frac{i-i_1}{i_2-i_1+1} \sum_{j=i}^{i_2} b_j\right) \geq 0,
    \end{aligned}
  \end{equation*}
  and the proof is completed.
\end{proof}

In the proof of Corollary~\ref{cor:rbX}, we shall use the following properties of the sticky particle dynamics.
\begin{lem}\label{lem:contract}
  The sticky particle dynamics has the following properties.
  \begin{itemize}
    \item {\em Flow:} Let $y^0 \in D_n$ and let us denote by $(\xi(t))_{t \geq 0}$ the sticky particle process started at $y^0$, with initial velocity vector $b$. For a given $\delta \geq 0$, let us denote by $(\xi'(s))_{s \geq 0}$ the sticky particle process started at $\xi(\delta)$, with initial velocity vector $b$. Then, for all $s \geq 0$, $\xi(\delta+s) = \xi'(s)$.
    
    \item {\em Contractivity:} Let $y^0, y'^0 \in D_n$ and let us denote by $(\xi(t))_{t \geq 0}$ and $(\xi'(t))_{t \geq 0}$ the sticky particle processes respectively started at $y^0$ and $y'^0$, with the same initial velocity vector $b$. Then, for all $t \geq 0$,
    \begin{equation*}
      \sum_{i=1}^n |\xi_i(t) - \xi'_i(t)| \leq \sum_{i=1}^n |y^0_i - y'^0_i|.
    \end{equation*}
  \end{itemize}
\end{lem}
\begin{proof}
  The flow property is a straighforward consequence of the definition of the sticky particle dynamics. Let us address the contractivity property. In this purpose, we write
  \begin{equation*}
    \forall t \geq 0, \qquad \xi(t) = y^0 + bt + \kappa(t), \quad \xi'(t) = y'^0 + bt + \kappa'(t),
  \end{equation*}
  so that, for all $t \geq 0$,
  \begin{equation*}
    \begin{aligned}
      \sum_{i=1}^n |\xi_i(t) - \xi'_i(t)| & = \sum_{i=1}^n |y^0_i - y'^0_i|\\
      & \quad + \sum_{i=1}^n \int_{s=0}^t \sgn(\xi_i(s) - \xi'_i(s))\dd \kappa_i(s) + \sum_{i=1}^n \int_{s=0}^t \sgn(\xi'_i(s) - \xi_i(s))\dd \kappa'_i(s),
    \end{aligned}
  \end{equation*}
  where $\sgn(\cdot)$ is defined by
  \begin{equation*}
    \sgn(x) = \left\{\begin{aligned}
      & 1 & \text{if $x>0$},\\
      & 0 & \text{if $x=0$},\\
      & -1 & \text{if $x<0$}.
    \end{aligned}\right.
  \end{equation*}

  We prove that
  \begin{equation*}
    \sum_{i=1}^n \int_{s=0}^t \sgn(\xi_i(s) - \xi'_i(s))\dd \kappa_i(s) \leq 0,
  \end{equation*}
  and the same arguments also yield 
  \begin{equation*}
    \sum_{i=1}^n \int_{s=0}^t \sgn(\xi'_i(s) - \xi_i(s))\dd \kappa'_i(s) \leq 0,
  \end{equation*}
  which completes the proof.
  
  With the notations of Lemma~\ref{lem:xi},
  \begin{equation*}
    \begin{aligned}
      & \sum_{i=1}^n \int_{s=0}^t \sgn(\xi_i(s) - \xi'_i(s))\dd \kappa_i(s) = \int_{s=0}^t \left(\sum_{i=1}^n \sgn(\xi_i(s) - \xi'_i(s))(\gamma_i(s)-\gamma_{i+1}(s))\right)\dd |\kappa|(s)\\
      & \qquad = \int_{s=0}^t \left(\sum_{i=2}^n \left(\sgn(\xi_i(s) - \xi'_i(s)) - \sgn(\xi_{i-1}(s) - \xi'_{i-1}(s))\right) \gamma_i(s)\right) \dd |\kappa|(s),
    \end{aligned}
  \end{equation*}
  where we have used Abel's transform as well as the fact that, $\dd|\kappa|(s)$-almost everywhere, $\gamma_1(s) = \gamma_{n+1}(s) = 0$. Now, $\dd|\kappa|(s)$-almost everywhere, either $\gamma_i(s) = 0$ or $\gamma_i(s) > 0$, in which case $\xi_i(s) = \xi_{i-1}(s)$ and therefore $\sgn(\xi_i(s) - \xi'_i(s)) - \sgn(\xi_{i-1}(s) - \xi'_{i-1}(s)) \leq 0$ since $\xi'_{i-1}(s) \leq \xi'_i(s)$. As a conclusion, 
  \begin{equation*}
    \int_{s=0}^t \left(\sum_{i=2}^n \left(\sgn(\xi_i(s) - \xi'_i(s)) - \sgn(\xi_{i-1}(s) - \xi'_{i-1}(s))\right) \gamma_i(s)\right) \dd |\kappa|(s) \leq 0,
  \end{equation*}
  and the proof is completed.
\end{proof}

\subsection{Small noise limit of the original particle system}\label{ss:rbX} Proposition~\ref{prop:contredsr} describes the small noise limit of the reordered particle system $Y^{\epsilon}$. We now describe the small noise limit of the original particle system $X^{\epsilon}$. For all $\sigma \in S_n$, we denote by $\xi_{\sigma^{-1}}$ the process $(\xi_{\sigma^{-1}(1)}(t), \ldots, \xi_{\sigma^{-1}(n)}(t))_{t \geq 0}$.

Recall that, for all $x \in \R^n$, $\bar{\Sigma}x$ refers to the set of permutations $\sigma \in S_n$ such that $x_{\sigma(1)} \leq \cdots \leq x_{\sigma(n)}$. When at least two particles have the same initial position, \ie $x^0 \in O_n$, $\bar{\Sigma}x^0$ contains more than one element. However, if each group of particles sharing the same initial position forms a single cluster in the sticky particle dynamics, then, for all $\sigma, \sigma' \in \bar{\Sigma}x^0$, the processes $\xi_{\sigma^{-1}}$ and $\xi_{\sigma'^{-1}}$ are equal.

\begin{cor}\label{cor:rbX}
  The small noise limit of the original particle system is described as follows.
  \begin{enumerate}
    \item If $\bar{\Sigma}x^0$ contains a single element, or if, for all $\sigma, \sigma' \in \bar{\Sigma}x^0$, the processes $\xi_{\sigma^{-1}}$ and $\xi_{\sigma'^{-1}}$ are equal, then $X^{\epsilon}$ converges in $L^2_{\loc}(\Pr_{x^0})$ to $\xi_{\sigma^{-1}}$ for any $\sigma \in \bar{\Sigma}x^0$.
    \item In general, $X^{\epsilon}$ converges in distribution to the process $\xi_{\sigma^{-1}}$, where $\sigma$ is a uniform random variable among $\bar{\Sigma}x^0$.
  \end{enumerate}
\end{cor}

Once again, by the same arguments as in Remark~\ref{rk:blum}, in the first case above, the convergence can be stated in $L^p_{\loc}(\Pr_{x^0})$, for all $p \in [1,+\infty)$, while if there exist at least $\sigma, \sigma' \in \bar{\Sigma}x^0$ such that $\xi_{\sigma^{-1}} \not= \xi_{\sigma'^{-1}}$, then in the second case above, the convergence cannot hold in probability.

\begin{proof}[Proof of Corollary~\ref{cor:rbX}]
  For all $T>0$ and $\alpha > 0$, let $B_T(\xi,\alpha)$ refer to the set of continuous paths $y \in C([0,T],D_n)$ such that $\sup_{t \in [0,T]} \max_{1 \leq i \leq n} |y_i(t)-\xi_i(t)| < \alpha$. Owing to Proposition~\ref{prop:contredsr}, for all $\alpha > 0$, $\lim_{\epsilon \dto 0} \Pr_{x^0}(Y^{\epsilon} \in B_T(\xi,\alpha)) = 1$.
  
  \sk
  Let us address the first part of the corollary. Let $\sigma$ be a fixed permutation in $\bar{\Sigma}x^0$. Note that, for all $i \in \{1, \ldots, n\}$, $X^{\epsilon}_{\sigma(i)}(0) = Y^{\epsilon}_i(0) = y^0_i$. Besides, for all $t \geq 0$, the definition of $X^{\epsilon}(t)$ yields
  \begin{equation*}
    \forall i \in \{1, \ldots, n\}, \qquad |X^{\epsilon}_{\sigma(i)}(t) - y^0_i| \leq \max_{1 \leq k \leq n} |b_k| t + \sqrt{2\epsilon} |W_{\sigma(i)}(t)|,
  \end{equation*}
  while the definition of $\xi(t)$ yields
  \begin{equation*}
    \forall i \in \{1, \ldots, n\}, \qquad |\xi_i(t) - y^0_i| \leq \max_{1 \leq k \leq n} |b_k| t.
  \end{equation*}
  As a consequence, for all $i \in \{1, \ldots, n\}$,
  \begin{equation*}
    \begin{aligned}
      |X^{\epsilon}_{\sigma(i)}(t) - \xi_i(t)|^2 & \leq 2\left(\left(\max_{1 \leq k \leq n} |b_k| t + \sqrt{2\epsilon} |W_{\sigma(i)}(t)|\right)^2 + \left(\max_{1 \leq k \leq n} |b_k| t\right)^2\right)\\
      & \leq 6 \max_{1 \leq k \leq n} (b_kt)^2 + 8\epsilon (W_{\sigma(i)}(t))^2.
    \end{aligned}
  \end{equation*}
  Therefore, for a fixed $T>0$,
  \begin{equation*}
    \sup_{t \in [0,T]} \sum_{i=1}^n |X^{\epsilon}_{\sigma(i)}(t) - \xi_i(t)|^2 \leq \sum_{i=1}^n \left(6 \max_{1 \leq k \leq n} (b_kT)^2 + 8\epsilon \sup_{t \in [0,T]} (W_{\sigma(i)}(t))^2\right),
  \end{equation*}
  so that 
  \begin{equation*}
    \begin{aligned}
      & \Exp_{x^0}\left(\sup_{t \in [0,T]} \sum_{i=1}^n |X^{\epsilon}_{\sigma(i)}(t) - \xi_i(t)|^2\ind{Y^{\epsilon} \not\in B_T(\xi,\alpha)}\right)\\
      & \qquad \leq 6n \max_{1 \leq k \leq n} (b_kT)^2\Pr_{x^0}(Y^{\epsilon} \not\in B_T(\xi,\alpha)) + 8n\epsilon \Exp_{x^0}\left(\sup_{t \in [0,T]} |W_1(t)|^2\right),
    \end{aligned}
  \end{equation*}
  therefore
  \begin{equation}\label{eq:pasBT}
    \forall \alpha > 0, \qquad \lim_{\epsilon \dto 0} \Exp_{x^0}\left(\sup_{t \in [0,T]} \sum_{i=1}^n |X^{\epsilon}_{\sigma(i)}(t) - \xi_i(t)|^2\ind{Y^{\epsilon} \not\in B_T(\xi,\alpha)}\right) = 0.
  \end{equation}
  We now fix $\eta > 0$ such that, for all $m \in \{0, \ldots, M-1\}$, $t^m < t^{m+1} - \eta$, and for all $m \in \{0, \ldots, M\}$, we denote by $I^m_{\eta}$ the interval $[0 \vee (t^m - \eta), (t^{m+1}-\eta) \wedge T]$. Then, one can choose $\alpha > 0$ small enough such that, for all $m \in \{0, \ldots, M\}$, for all $i,j \in \{1, \ldots, n\}$ such that $i < j$ and $i \not\sim_m j$,
  \begin{equation}\label{eq:BTxialpha}
    \sup_{t \in I^m_{\eta}} \xi_i(t) + \alpha < \inf_{t \in I^m_{\eta}}\xi_j(t) - \alpha,
  \end{equation}
  see Figure~\ref{fig:BTxialpha:1}. In particular, if $y=(y_1, \ldots, y_n) \in B_T(\xi,\alpha)$ and $t \in I^m_{\eta}$ is such that $y_i(t)=y_j(t)$, then $i \sim_m j$. Here, it is crucial that either all the particles have pairwise distinct initial positions, or that each group of particles sharing the same initial position forms a single cluster in the sticky particle dynamics. Otherwise, for all $\alpha > 0$,~\eqref{eq:BTxialpha} would fail for $m=0$.
  
  \begin{figure}[ht]
    \begin{center}
      \includegraphics{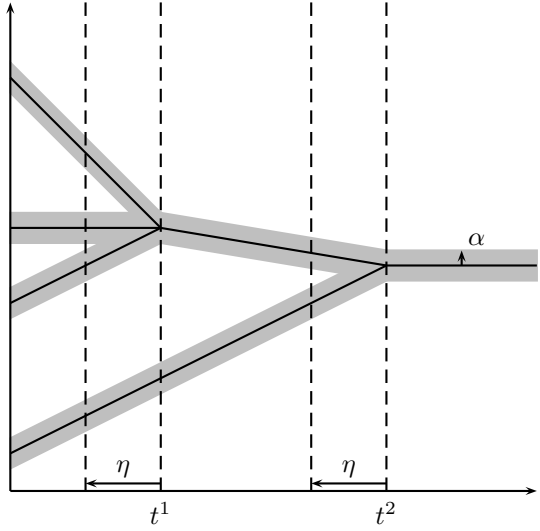}
      \caption{A trajectory of the sticky particle dynamics $\xi$ for $n=4$ particles, with $M=2$ collisions. The initial positions of the particles are pairwise distinct. For $\eta > 0$ such that $0 < t^1- \eta < t^1 < t^2 - \eta$, $\alpha > 0$ is chosen small enough for the set $B_T(\xi,\alpha)$ to satisfy the condition~\eqref{eq:BTxialpha}. A path $y$ in $B_T(\xi,\alpha)$ is necessarily contained in the gray area.}
      \label{fig:BTxialpha:1}
    \end{center}
  \end{figure}
  
  Such a choice for $\alpha$ ensures the following assertion:
  \begin{itemize}
    \item[($*$)] If $\alpha > 0$ satisfies~\eqref{eq:BTxialpha}, then on the event $\{Y^{\epsilon} \in B_T(\xi,\alpha)\}$, for all $m \in \{0, \ldots, M\}$, for all $t \in I^m_{\eta}$, for all $i,j \in \{1, \ldots, n\}$ such that $X^{\epsilon}_{\sigma(i)}(t) = Y^{\epsilon}_j(t)$, then $i \sim_m j$.
  \end{itemize}
  Before proving ($*$), let us show how this assertion allows to conclude: for all $t \in [0,T]$, there exists $m \in \{0, \ldots, M\}$ such that $t \in I^m_{\eta}$. Let us fix $i \in \{1, \ldots, n\}$ and $j$ such that $X^{\epsilon}_{\sigma(i)}(t) = Y^{\epsilon}_j(t)$. Then, by~($*$), $j \sim_m i$. On the event $\{Y^{\epsilon} \in B_T(\xi,\alpha)\}$,
  \begin{itemize}
    \item if $t \in [t^m, (t^{m+1}-\eta)\wedge T]$, then $\xi_j(t)=\xi_i(t)$, so that $|X^{\epsilon}_{\sigma(i)} - \xi_i(t)| = |Y^{\epsilon}_j(t)-\xi_j(t)| < \alpha$;
    \item if $m \geq 1$ and $t \in [t^m-\eta, t^m\wedge T]$, then $|\xi_j(t)-\xi_i(t)| = |\xi_j(t^m) - v^m_j(t^m-t) - \xi_i(t^m) + v^m_i(t^m-t)| \leq 2 \max_{1 \leq k \leq n} |b_k| \eta$, so that $|X^{\epsilon}_{\sigma(i)}(t) - \xi_i(t)| \leq |Y^{\epsilon}_j(t) - \xi_j(t)| + |\xi_j(t)-\xi_i(t)| < \alpha + 2 \max_{1 \leq k \leq n} |b_k| \eta$.
  \end{itemize}
  As a conclusion,
  \begin{equation*}
    \sup_{t \in [0,T]} \sum_{i=1}^n |X^{\epsilon}_{\sigma(i)}(t) - \xi_i(t)|^2\ind{Y^{\epsilon} \in B_T(\xi,\alpha)} \leq n\left(\alpha + 2 \max_{1 \leq k \leq n} |b_k| \eta\right)^2.
  \end{equation*}
  Taking the expectation of both sides above, recalling~\eqref{eq:pasBT}, letting $\epsilon \dto 0$, $\alpha \dto 0$ and finally $\eta \dto 0$, we conclude that
  \begin{equation*}
    \lim_{\epsilon \dto 0} \Exp_{x^0}\left(\sup_{t \in [0,T]} \sum_{i=1}^n |X^{\epsilon}_{\sigma(i)}(t) - \xi_i(t)|^2\right) = 0.
  \end{equation*}
  
  \sk
  Before addressing the second part of the corollary, let us prove the assertion~($*$). Let us assume that $\alpha > 0$ satisfies~\eqref{eq:BTxialpha} and that $Y^{\epsilon} \in B_T(\xi,\alpha)$. Let $m \in \{0, \ldots, M\}$, $t \in I^m_{\eta}$ and $i,j \in \{1, \ldots, n\}$ such that $X^{\epsilon}_{\sigma(i)}(t) = Y^{\epsilon}_j(t)$. If $i=j$, then there is nothing to prove. Let us assume that $i<j$, the arguments for the case $i>j$ being symmetric. By the continuity of the trajectories of $X^{\epsilon}_1, \ldots, X^{\epsilon}_n$ and the fact that $X^{\epsilon}_{\sigma(i)}(0)=Y^{\epsilon}_i(0)$, there exists a nondecreasing sequence of times $0 \leq t_{i,i+1} \leq \cdots \leq t_{j-1,j} \leq t$ such that, for all $k \in \{i, \ldots, j-1\}$, $X^{\epsilon}_{\sigma(i)}(t_{k,k+1}) = Y^{\epsilon}_k(t_{k,k+1}) = Y^{\epsilon}_{k+1}(t_{k,k+1})$. Certainly, there is an associated nondecreasing sequence of integers $0 \leq m_{i,i+1} \leq \cdots \leq m_{j-1,j} \leq m$ such that, for all $k \in \{i, \ldots, j-1\}$, $t_{k,k+1} \in I^{m_{k,k+1}}_{\eta}$. By~\eqref{eq:BTxialpha}, for all $k \in \{i, \ldots, j-1\}$, $k \sim_{m_{k,k+1}} k+1$, and since $m_{k,k+1} \leq m$, then $k \sim_m k+1$. Due to the transitivity of the relation $\sim_m$, we conclude that $i \sim_m i+1 \sim_m \cdots \sim_m j$.

  \sk
  Let us now address the second part of the corollary. Let us fix $T>0$, $\delta > 0$ such that $\delta < t^1 \wedge T$ and $\eta > 0$ such that $\delta < t^1 - \eta$, and for $m \in \{1, \ldots, M-1\}$, $t^m < t^{m+1} - \eta$. We slightly modify~\eqref{eq:BTxialpha} as, for all $m \in \{0, \ldots, M\}$, we denote by $I^m_{\eta,\delta}$ the interval $[\delta \vee (t^m - \eta), (t^{m+1}-\eta) \wedge T]$ and we choose $\alpha > 0$ small enough such that, for all $m \in \{0, \ldots, M\}$, for all $i,j \in \{1, \ldots, n\}$ such that $i < j$ and $i \not\sim_m j$,
  \begin{equation}\label{eq:BTxialpha:2}
    \sup_{t \in I^m_{\eta,\delta}} \xi_i(t) + \alpha < \inf_{t \in I^m_{\eta,\delta}}\xi_j(t) - \alpha,
  \end{equation}
  see Figure~\ref{fig:BTxialpha:2}. In particular, if $y=(y_1, \ldots, y_n) \in B_T(\xi,\alpha)$ and $t \in I^m_{\eta,\delta}$ is such that $y_i(t)=y_j(t)$, then $i \sim_m j$; while, if $t \in [0,\delta]$ is such that $y_i(t)=y_j(t)$, then $\xi_i(0) = \xi_j(0)$ although the relation $i \sim_0 j$ does not necessarily hold.
  
  \begin{figure}[ht]
    \begin{center}
      \includegraphics{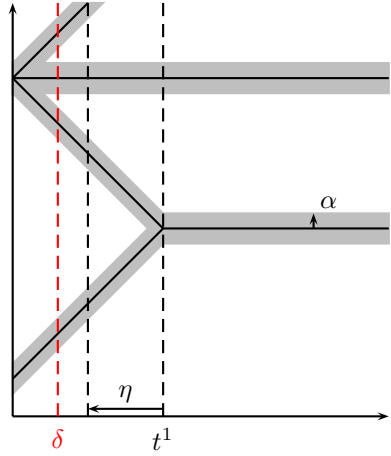}
      \caption{If some particles share the same initial position but instantaneously split into several clusters, $\delta$ is fixed in $(0,t^1 \wedge T)$ and $\eta$, $\alpha$ are taken small enough for~\eqref{eq:BTxialpha:2} to hold.}
      \label{fig:BTxialpha:2}
    \end{center}
  \end{figure}
  
  Now let $F : C([0,T], \R^n) \to \R$ be a bounded and Lipschitz continuous function, with unit Lipschitz norm. We shall prove that 
  \begin{equation*}
    \lim_{\epsilon \dto 0} \Exp_{x^0}(F(X^{\epsilon})) = \frac{1}{|\bar{\Sigma}x^0|} \sum_{\sigma \in \bar{\Sigma}x^0} F(\xi_{\sigma^{-1}}),
  \end{equation*}
  which leads to the second part of the lemma on account of the Portmanteau theorem~\cite[Theorem~2.1, p.~11]{billingsley}.
  
  First, by the boundedness of $F$, $\lim_{\epsilon \dto 0} \Exp_{x^0}\left(F(X^{\epsilon})\ind{Y^{\epsilon} \not\in B_T(\xi,\alpha)}\right) = 0$. Second,
  \begin{equation*}
    \begin{aligned}
      \Exp_{x^0}\left(F(X^{\epsilon})\ind{Y^{\epsilon} \in B_T(\xi,\alpha)}\right) & = \sum_{\sigma \in S_n}\Exp_{x^0}\left(F(X^{\epsilon})\ind{Y^{\epsilon} \in B_T(\xi,\alpha), \Sigma X^{\epsilon}(\delta) = \sigma}\right)\\
      & = \sum_{\sigma \in \bar{\Sigma} x^0}\Exp_{x^0}\left(F(X^{\epsilon})\ind{Y^{\epsilon} \in B_T(\xi,\alpha), \Sigma X^{\epsilon}(\delta) = \sigma}\right),
    \end{aligned}
  \end{equation*}
   as, on the event $\{Y^{\epsilon} \in B_T(\xi,\alpha)\}$, the continuity of the trajectories of $X^{\epsilon}_1, \ldots, X^{\epsilon}_n$ as well as the choice of $\delta$ and $\alpha$ imply that $\Sigma X^{\epsilon}(\delta) \in \bar{\Sigma}x^0$. As a consequence,
   \begin{equation*}
     \begin{aligned}
       & \left|\Exp_{x^0}\left(F(X^{\epsilon})\ind{Y^{\epsilon} \in B_T(\xi,\alpha)}\right) - \frac{1}{|\bar{\Sigma}x^0|} \sum_{\sigma \in \bar{\Sigma}x^0} F(\xi_{\sigma^{-1}})\right|\\
       & \qquad \leq \sum_{\sigma \in \bar{\Sigma}x^0} \left|\Exp_{x^0}\left(F(X^{\epsilon})\ind{Y^{\epsilon} \in B_T(\xi,\alpha), \Sigma X^{\epsilon}(\delta) = \sigma} - \frac{1}{|\bar{\Sigma}x^0|}F(\xi_{\sigma^{-1}})\right)\right|,
     \end{aligned}
   \end{equation*}
   and, for all $\sigma \in \bar{\Sigma}x^0$,
   \begin{equation}\label{eq:ExpFX}
     \begin{aligned}
       & \left|\Exp_{x^0}\left(F(X^{\epsilon})\ind{Y^{\epsilon} \in B_T(\xi,\alpha), \Sigma X^{\epsilon}(\delta) = \sigma} - \frac{1}{|\bar{\Sigma}x^0|}F(\xi_{\sigma^{-1}})\right)\right|\\
       & \qquad \leq \Exp_{x^0}\left(\left|F(X^{\epsilon})-F(\xi_{\sigma^{-1}})\right|\ind{Y^{\epsilon} \in B_T(\xi,\alpha), \Sigma X^{\epsilon}(\delta) = \sigma}\right)\\
       & \qquad \quad + ||F||_{\infty}\left|\Pr_{x^0}(Y^{\epsilon} \in B_T(\xi,\alpha), \Sigma X^{\epsilon}(\delta) = \sigma) - \frac{1}{|\bar{\Sigma}x^0|}\right|.
     \end{aligned}
   \end{equation}
   
   Let us prove that the first term in the right-hand side of~\eqref{eq:ExpFX} vanishes. The Lipschitz continuity of $F$ yields
   \begin{equation*}
     \begin{aligned}
       & \Exp_{x^0}\left(\left|F(X^{\epsilon})-F(\xi_{\sigma^{-1}})\right|\ind{Y^{\epsilon} \in B_T(\xi,\alpha), \Sigma X^{\epsilon}(\delta) = \sigma}\right)\\
       & \qquad \leq \Exp_{x^0}\left(\sup_{t \in [0,T]} \sum_{i=1}^n |X^{\epsilon}_{\sigma(i)}(t)-\xi_i(t)|\ind{Y^{\epsilon} \in B_T(\xi,\alpha), \Sigma X^{\epsilon}(\delta) = \sigma}\right).
     \end{aligned}
   \end{equation*}
   By~\eqref{eq:BTxialpha:2}, for all $t \in [0,\delta]$, for all $i,j \in \{1, \ldots, n\}$ such that $X^{\epsilon}_{\sigma(i)}(t) = Y^{\epsilon}_j(t)$, then $|X^{\epsilon}_{\sigma(i)}(t) - \xi_i(t)| \leq |Y^{\epsilon}_j(t) - \xi_j(t)| + |\xi_j(t) - \xi_i(t)| < \alpha + 2\max_{1 \leq k \leq n} |b_i| \delta$ if $Y^{\epsilon} \in B_T(\xi,\alpha)$. As a consequence,
   \begin{equation*}
     \Exp_{x^0}\left(\sup_{t \in [0,\delta]} \sum_{i=1}^n |X^{\epsilon}_{\sigma(i)}(t)-\xi_i(t)|^2\ind{Y^{\epsilon} \in B_T(\xi,\alpha), \Sigma X^{\epsilon}(\delta) = \sigma}\right) \leq n(\alpha + 2\max_{1 \leq k \leq n} |b_k| \delta)^2,
   \end{equation*}
   which vanishes when $\alpha \dto 0$ and $\delta \dto 0$.
   
   Besides, 
   \begin{equation*}
     \begin{aligned}
       & \Exp_{x^0}\left(\sup_{t \in [\delta,T]} \sum_{i=1}^n |X^{\epsilon}_{\sigma(i)}(t)-\xi_i(t)|\ind{Y^{\epsilon} \in B_T(\xi,\alpha), \Sigma X^{\epsilon}(\delta) = \sigma}\right)\\
       & \qquad = \Exp_{x^0}\left(\ind{\Sigma X^{\epsilon}(\delta) = \sigma}\Exp_{x^0}\left(\left.\sup_{s \in [0,T-\delta]} \sum_{i=1}^n |X^{\epsilon}_{\sigma(i)}(\delta+s)-\xi_i(\delta+s)|\ind{Y^{\epsilon} \in B_T(\xi,\alpha)}\right|\mathcal{F}_{\delta}\right)\right)\\  
       & \qquad \leq \Exp_{x^0}\left(\ind{\Sigma X^{\epsilon}(\delta) = \sigma}\Exp_{x^0}\left(\left.\sup_{s \in [0,T-\delta]} \sum_{i=1}^n |X^{\epsilon}_{\sigma(i)}(\delta+s)-\xi'_i(s)|\ind{Y^{\epsilon} \in B_T(\xi,\alpha)}\right|\mathcal{F}_{\delta}\right)\right)\\  
       & \qquad \quad + \Exp_{x^0}\left(\ind{\Sigma X^{\epsilon}(\delta) = \sigma}\Exp_{x^0}\left(\left.\sup_{s \in [0,T-\delta]} \sum_{i=1}^n |\xi'_i(s)-\xi_i(\delta+s)|\ind{Y^{\epsilon} \in B_T(\xi,\alpha)}\right|\mathcal{F}_{\delta}\right)\right),
     \end{aligned}
   \end{equation*}
   where, on the event $\{\Sigma X^{\epsilon}(\delta) = \sigma\} \in \mathcal{F}_{\delta}$, $(\xi'(s))_{s \geq 0}$ refers to the sticky particle process started at $(X^{\epsilon}_{\sigma(1)}(\delta), \ldots, X^{\epsilon}_{\sigma(n)}(\delta)) \in D_n$, with initial velocity vector $b$. On the one hand, on the event $\{\Sigma X^{\epsilon}(\delta) = \sigma\}$, Lemma~\ref{lem:contract} yields
   \begin{equation*}
     \begin{aligned}
       & \Exp_{x^0}\left(\left.\sup_{s \in [0,T-\delta]} \sum_{i=1}^n |\xi'_i(s)-\xi_i(\delta+s)|\ind{Y^{\epsilon} \in B_T(\xi,\alpha)}\right|\mathcal{F}_{\delta}\right)\\
       & \qquad \leq \Exp_{x^0}\left(\left.\sum_{i=1}^n |X^{\epsilon}_{\sigma(i)}(\delta)-\xi_i(\delta)|\ind{Y^{\epsilon} \in B_T(\xi,\alpha)}\right|\mathcal{F}_{\delta}\right)\\
       & \qquad \leq \Exp_{x^0}\left(\left.\sum_{i=1}^n |Y^{\epsilon}_i(\delta)-\xi_i(\delta)|\ind{Y^{\epsilon} \in B_T(\xi,\alpha)}\right|\mathcal{F}_{\delta}\right)\\
       & \qquad \leq n\alpha,
     \end{aligned}     
   \end{equation*}
   since the choice of $\delta$ ensures that, on the event $\{Y^{\epsilon} \in B_T(\xi,\alpha), \Sigma X^{\epsilon}(\delta) = \sigma\}$, $X^{\epsilon}_{\sigma(i)}(\delta) = Y^{\epsilon}_i(\delta)$ for all $i \in \{1, \ldots, n\}$. On the other hand, let $X'^{\epsilon}(s) := X^{\epsilon}(\delta+s)$ for all $s \geq 0$, and let $Y'^{\epsilon}$ be defined accordingly. Then
   \begin{equation*}
     \begin{aligned}
       & \ind{\Sigma X^{\epsilon}(\delta) = \sigma}\Exp_{x^0}\left(\left.\sup_{s \in [0,T-\delta]} \sum_{i=1}^n |X^{\epsilon}_{\sigma(i)}(\delta+s)-\xi'_i(s)|\ind{Y^{\epsilon} \in B_T(\xi,\alpha)}\right|\mathcal{F}_{\delta}\right)\\
       & \qquad \leq \ind{\Sigma X'^{\epsilon}(0) = \sigma}\Exp_{x^0}\left(\left.\sup_{s \in [0,T-\delta]} \sum_{i=1}^n |X'^{\epsilon}_{\sigma(i)}(s)-\xi'_i(s)|\ind{Y'^{\epsilon} \in B_{T-\delta}(\xi,\alpha)}\right|\mathcal{F}_{\delta}\right).
     \end{aligned}
   \end{equation*}
   By Proposition~\ref{prop:X}, $\Pr_{x^0}$-almost surely on the event $\{\Sigma X^{\epsilon}(\delta)=\sigma\}$, $\sigma$ is the only element of $\bar{\Sigma} X^{\epsilon}(\delta) = \bar{\Sigma} X'^{\epsilon}(0)$. Therefore, combining the Markov property with the first part of the proof,we obtain that
   \begin{equation*}
     \begin{aligned}
       & \ind{\Sigma X'^{\epsilon}(0) = \sigma}\Exp_{x^0}\left(\left.\sup_{s \in [0,T-\delta]} \sum_{i=1}^n |X'^{\epsilon}_{\sigma(i)}(s)-\xi'_i(s)|\ind{Y'^{\epsilon} \in B_{T-\delta}(\xi,\alpha)}\right|\mathcal{F}_{\delta}\right)\\
       & \qquad \leq n\left(\alpha + 2 \max_{1 \leq k \leq n} |b_k| \eta\right)^2.
     \end{aligned}
   \end{equation*}
   As a conclusion, the right-hand side of~\eqref{eq:ExpFX} vanishes when $\epsilon \dto 0$, $\alpha \dto 0$ and $\eta \dto 0$.
   
   We now address the second term in the right-hand side of~\eqref{eq:ExpFX}. For all $\sigma \in \bar{\Sigma}x^0$, the process $(X^{\epsilon}_{\sigma(1)}, \ldots, X^{\epsilon}_{\sigma(n)})$ solves the stochastic differential equation
   \begin{equation*}
     \forall t \geq 0, \qquad X^{\epsilon}_{\sigma(i)}(t) = y^0_i + \sum_{j=1}^n \int_{s=0}^t \ind{X^{\epsilon}_{\sigma(i)}(s) = Y^{\epsilon}_j(s)} b_j \dd s + \sqrt{2\epsilon} W_{\sigma(i)}(t).
   \end{equation*}
   Since, for all $\sigma \in \bar{\Sigma}x^0$, $(W_{\sigma(1)}, \ldots, W_{\sigma(n)})$ is a standard Brownian motion, the uniqueness in law for the solutions to the equation above (due to the Girsanov theorem or as a consequence of Proposition~\ref{prop:X} combined with the Yamada-Watanabe theorem) implies that the processes $(X^{\epsilon}_{\sigma(1)}, \ldots, X^{\epsilon}_{\sigma(n)})$ have the same distribution, for all $\sigma \in \bar{\Sigma}x^0$. As a consequence,
  \begin{equation*}
    \forall \sigma \in \bar{\Sigma}x^0, \qquad \Pr_{x^0}(Y^{\epsilon} \in B_T(\xi,\alpha), \Sigma X^{\epsilon}(\delta) = \sigma) = \frac{1}{|\bar{\Sigma}x^0|}\Pr_{x^0}(Y^{\epsilon} \in B_T(\xi,\alpha)),
  \end{equation*}
  therefore
  \begin{equation*}
    \lim_{\epsilon \dto 0} \Pr_{x^0}(Y^{\epsilon} \in B_T(\xi,\alpha), \Sigma X^{\epsilon}(\delta) = \sigma) = \frac{1}{|\bar{\Sigma}x^0|},
  \end{equation*}
  and the second term in the right-hand side of~\eqref{eq:ExpFX} vanishes when $\epsilon \dto 0$. Letting $\epsilon \dto 0$, $\alpha \dto 0$, $\eta \dto 0$ and, finally, $\delta \dto 0$, we conclude that
  \begin{equation*}
    \lim_{\epsilon \dto 0} \Exp_{x^0}(F(X^{\epsilon})) = \frac{1}{|\bar{\Sigma}x^0|} \sum_{\sigma \in \bar{\Sigma}x^0} F(\xi_{\sigma^{-1}}),
  \end{equation*}
  which completes the proof.
\end{proof}

%%%%%%%%%%%%%%%%%%%%%%%%%%%%%%%%%%%%%%%%%%%%%%%%%%%%%%%%%%%%%%%%%%%%%%%%%%%%%%%%%%%%%%%%%%%%%%%%%%%%%%%%%%%%%%%%%%%%%%%

\section{The order-based case}\label{s:clus}

We now address the general case of order-based processes. If the initial condition $x^0 \in \R^n$ is such that the particles have pairwise distinct initial positions, \ie $x^0 \not\in O_n$, by the same arguments as in the two-particle case of Section~\ref{s:n2}, in the small noise limit the $i$-th particle travels at constant velocity $b_i(\Sigma x^0)$ until the first collision in the system. Thus, the problem is reduced to the case of initial conditions $x^0 \in O_n$ for which several particles have the same position. In this case, the isolated particles have no influence on the instantaneous behaviour of the system, as they cannot immediately collide with other particles. Up to decreasing the number of particles, the problem can be reduced to the case of initial conditions where there are no isolated particles. Still, the interactions inside each group of particles with the same initial position are likely to modify the drifts of the particles in the other groups. In this section, we avoid such situations and assume that all the particles in the system share the same initial position. Since the function $\Sigma$ is invariant by translation, there is no loss of generality in taking $x^0=0$. 

In Subsection~\ref{ss:stab}, we provide an extension of the stability condition of Lemma~\ref{lem:stab} for the rank-based case, which ensures that the particles aggregate into a single cluster in the small noise limit. We describe the motion of this cluster under a slighlty stronger stability condition in Subsection~\ref{ss:cluster}. Finally, in Subsection~\ref{ss:contrex}, we exhibit the example of a system with three particles for which the particles aggregate into a single cluster in the small noise limit, although the stability condition is not satisfied.

\subsection{The stability condition}\label{ss:stab} In the rank-based case addressed in Section~\ref{s:rank}, it is observed that, in the small noise limit, if all the particles stick into a single cluster, then the velocities satisfy the {\em stability condition} that for any partition of the set $\{1, \ldots, n\}$ into a leftmost subset $\{1, \ldots, i\}$ and a rightmost subset $\{i+1, \ldots, n\}$, the average velocity of the group of leftmost particles is larger than the average velocity of the group of rightmost particles (see Lemma~\ref{lem:stab}).

The purpose of this subsection is to extend this stability condition to general order-based drift functions $b$. More precisely, the function $b : S_n \to \R^n$ is said to satisfy the stability condition~\eqref{eq:SC} if
\begin{equation}\label{eq:SC}
  \forall \sigma \in S_n, \quad \forall i \in \{1, \ldots, n-1\}, \qquad \frac{1}{i}\sum_{j=1}^i b_{\sigma(j)}(\sigma) \geq \frac{1}{n-i}\sum_{j=i+1}^n b_{\sigma(j)}(\sigma),
  \tag{SC}
\end{equation}
which has to be understood as the extension of the stability condition of Lemma~\ref{lem:stab} in Section~\ref{s:rank}.

\subsubsection{The projected system} Similarly to the two-particle case addressed in Section~\ref{s:n2}, in which the behaviour of $(X^{\epsilon}_1,X^{\epsilon}_2)$ heavily depends on the behaviour of the scalar process $Z^{\epsilon}=X^{\epsilon}_1-X^{\epsilon}_2$, the dimensionality of the problem can be reduced by subtracting the center of mass of the system to the positions of the particles. This amounts to considering the orthogonal projection $Z^{\epsilon} = (Z^{\epsilon}_1(t), \ldots, Z^{\epsilon}_n(t))_{t \geq 0}$ of $X^{\epsilon}$ on the hyperplane $M_n := \{(z_1, \ldots, z_n) \in \R^n : z_1 + \cdots + z_n = 0\}$. The orthogonal projection of $\R^n$ on $M_n$ is denoted by $\Pi$ and writes $\Pi = I_n - (1/n)J_n$, where $I_n$ is the identity matrix and $J_n$ refers to the matrix with all coefficients equal to $1$. Then, $Z^{\epsilon}$ is a diffusion on the hyperplane $M_n$ and satisfies
\begin{equation*}
  \forall t \geq 0, \qquad Z^{\epsilon}(t) = \int_{s=0}^t b^{\Pi}(\Sigma Z^{\epsilon}(s)) \dd s + \sqrt{2\epsilon} \Pi W(t),
\end{equation*}
where $b^{\Pi} := \Pi b$. Note that the stability condition~\eqref{eq:SC} rewrites
\begin{equation}\label{eq:SC:Pi}
  \forall \sigma \in S_n, \quad \forall i \in \{1, \ldots, n-1\}, \qquad \sum_{j=1}^i b^{\Pi}_{\sigma(j)}(\sigma) \geq 0.
\end{equation}

\subsubsection{Aggregation into a single cluster} In the small noise limit, all the particles $X^{\epsilon}_1, \ldots, X^{\epsilon}_n$ stick together into a single cluster if and only if $Z^{\epsilon}$ converges to $0$. This is ensured by the stability condition~\eqref{eq:SC}.

\begin{prop}\label{prop:singclus}
  Under the stability condition~\eqref{eq:SC}, for all $T>0$,
  \begin{equation*}
    \Exp_0\left(\sup_{t \in [0,T]} \sum_{i=1}^n |Z^{\epsilon}_i(t)|^2\right) \leq (4\sqrt{2}+2)(n-1)\epsilon T.
  \end{equation*}
\end{prop}
\begin{proof}
  By the Itô formula, for all $t \geq 0$,
  \begin{equation*}
    \sum_{i=1}^n |Z_i^{\epsilon}(t)|^2 = 2 \int_{s=0}^t \sum_{i=1}^n Z^{\epsilon}_i(s) b_i^{\Pi}(\Sigma Z^{\epsilon}(s))\dd s + 2\sqrt{2\epsilon} M^{\epsilon}(t) + 2\epsilon (n-1) t,
  \end{equation*}
  where
  \begin{equation*}
    M^{\epsilon}(t) := \sum_{i=1}^n \int_{s=0}^t Z^{\epsilon}_i(s) \dd W^{\Pi}_i(s), \qquad W^{\Pi}_i(t) := \left(1-\frac{1}{n}\right) W_i(t) - \frac{1}{n}\sum_{j\not=i} W_j(t).
  \end{equation*}
  
  Under the stability condition~\eqref{eq:SC}, let us fix $z=(z_1, \ldots, z_n) \in M_n$, $\sigma = \Sigma z$ and compute
  \begin{equation}\label{eq:SBP}
    \begin{aligned}
      \sum_{i=1}^n z_i b^{\Pi}_i(\sigma) = \sum_{i=1}^n z_{\sigma(i)} b^{\Pi}_{\sigma(i)}(\sigma) & = \sum_{i=1}^{n-1} b^{\Pi}_{\sigma(i)}(\sigma)\sum_{j=i}^{n-1}(z_{\sigma(j)}-z_{\sigma(j+1)})\\
      & = \sum_{j=1}^{n-1}(z_{\sigma(j)}-z_{\sigma(j+1)})\sum_{i=1}^j b^{\Pi}_{\sigma(i)}(\sigma),
    \end{aligned}
  \end{equation}
  where we have used the fact that $\sum_{j=1}^n b^{\Pi}_{\sigma(j)}(\sigma) = 0$ as $b^{\Pi}(\sigma) \in M_n$. For all $j \in \{1, \ldots, n-1\}$, the definition of $\sigma = \Sigma z$ yields $z_{\sigma(j)}-z_{\sigma(j+1)} \leq 0$ while $\sum_{i=1}^j b^{\Pi}_{\sigma(i)}(\sigma) \geq 0$ by~\eqref{eq:SC:Pi}. As a conclusion,
  \begin{equation*}
    \forall z \in M_n, \qquad \sum_{i=1}^n z_i b^{\Pi}_i(\Sigma z) \leq 0.
  \end{equation*}
  
  As a consequence, for all $t \geq 0$, $\sum_{i=1}^n |Z_i^{\epsilon}(t)|^2 \leq 2\sqrt{2\epsilon} M^{\epsilon}(t) + 2\epsilon (n-1) t$, and the result follows from the same localization procedure as in the proof of Proposition~\ref{prop:Z}, case~\eqref{case:concon:z} and the use of the Kunita-Watanabe inequality to estimate
  \begin{equation*}
    \begin{aligned}
      \Exp_0(M^{\epsilon}(T)^2) & = \Exp_0(\langle M^{\epsilon}\rangle(T)) = \sum_{i,j=1}^n \Exp_0\left(\int_{s=0}^T Z^{\epsilon}_i(s)Z^{\epsilon}_j(s)\dd \langle W^{\Pi}_i, W^{\Pi}_j\rangle(s)\right)\\
      & \leq \sum_{i,j=1}^n \Exp_0\left(\sqrt{\int_{s=0}^T Z^{\epsilon}_i(s)^2\dd \langle W^{\Pi}_i\rangle(s)}\sqrt{\int_{s=0}^T Z^{\epsilon}_j(s)^2\dd \langle W^{\Pi}_j\rangle(s)}\right)\\
      & \leq \sum_{i,j=1}^n \sqrt{\Exp_0\left(\int_{s=0}^T Z^{\epsilon}_i(s)^2\dd \langle W^{\Pi}_i\rangle(s)\right)}\sqrt{\Exp_0\left(\int_{s=0}^T Z^{\epsilon}_j(s)^2\dd \langle W^{\Pi}_j\rangle(s)\right)}\\
      & = \left(1-\frac{1}{n}\right)\left(\sum_{i=1}^n\sqrt{\Exp_0\left(\int_{s=0}^T Z^{\epsilon}_i(s)^2\dd s\right)}\right)^2\\
      & \leq (n-1) \Exp_0\left(\int_{s=0}^T \sum_{i=1}^n Z^{\epsilon}_i(s)^2\dd s\right),
    \end{aligned}
  \end{equation*}
  where we have used the Cauchy-Schwarz inequality at the third line.
\end{proof}

\subsection{Velocity of the cluster}\label{ss:cluster} According to Proposition~\ref{prop:singclus}, under the stability condition~\eqref{eq:SC}, the particles stick together and form a cluster in the small noise limit. The purpose of this subsection is to determine the motion of the cluster. 

\subsubsection{The strong stability condition} In the two-particle case of Section~\ref{s:n2}, the stability condition~\eqref{eq:SC} corresponds to the case of converging/converging configurations~\eqref{case:concon} and~\eqref{case:conconl} in Proposition~\ref{prop:concon}. In order to rule out degenerate situations such as case~\eqref{case:conconl}, in which the velocity of the two-particle cluster is random and nonconstant, we introduce the following {\em strong stability condition}:
\begin{equation}\label{eq:SSC}
  \forall \sigma \in S_n, \quad \forall i \in \{1, \ldots, n-1\}, \qquad \frac{1}{i}\sum_{j=1}^i b_{\sigma(j)}(\sigma) > \frac{1}{n-i}\sum_{j=i+1}^n b_{\sigma(j)}(\sigma).
  \tag{SSC}
\end{equation}
Similarly to~\eqref{eq:SC:Pi}, the strong stability condition~\eqref{eq:SSC} rewrites
\begin{equation}\label{eq:SSC:Pi}
  \bar{b} := \inf_{\sigma \in S_n} \inf_{1 \leq i \leq n-1} \sum_{j=1}^i b^{\Pi}_{\sigma(j)}(\sigma) > 0.
\end{equation}

\begin{lem}\label{lem:lyap}
  Under the strong stability condition~\eqref{eq:SSC}, for all $z=(z_1, \ldots, z_n) \in M_n$,
  \begin{equation*}
    \sum_{i=1}^n z_i b^{\Pi}_i(\Sigma z) \leq -\bar{b}\max_{1 \leq i \leq n} |z_i|.
  \end{equation*}
\end{lem}
\begin{proof}
  By~\eqref{eq:SBP} and~\eqref{eq:SSC:Pi}, $\sum_{i=1}^n z_i b^{\Pi}_i(\sigma) = \leq -\bar{b} (z_{\sigma(n)}-z_{\sigma(1)})$, where $\sigma = \Sigma z$. Since $z_{\sigma(1)} \leq \cdots \leq z_{\sigma(n)}$ and $z_1 + \cdots + z_n = 0$, it is an easy barycentric inequality that $z_{\sigma(n)}-z_{\sigma(1)} \geq \max_{1 \leq i \leq n} |z_i|$, and the proof is completed.
\end{proof}

\subsubsection{Changing the space-time scale} Let $X^1$ refer to the solution to 
\begin{equation*}
  \forall t \geq 0, \qquad X^1(t) = \int_{s=0}^t b(\Sigma X^1(s)) \dd s + \sqrt{2} W(t). 
\end{equation*}
In the rank-based case, the strong stability condition~\eqref{eq:SSC} was identified by Pal and Pitman~\cite[Remark, p.~2187]{pp} as a necessary and sufficient condition for the law of process $Z^1 = \Pi X^1$ to converge in total variation to its unique stationary distribution. In the order-based case, the interpretation of the small noise limit of $Z^{\epsilon}$ in terms of the long time behaviour of the process $Z^1$ can be made explicit through the following space-time scale change.

For all $\epsilon > 0$, let us define $\tilde{X}^{\epsilon}(t) := \epsilon X^1(t/\epsilon)$. Then it is straightforward to check that there exists a standard Brownian motion $\tilde{W}^{\epsilon}$ in $\R^n$ on $(\Omega, \mathcal{F}, \Pr_0)$ such that
\begin{equation*}
  \forall t \geq 0, \qquad \tilde{X}^{\epsilon}(t) = \int_{s=0}^t b(\Sigma \tilde{X}^{\epsilon}(s)) \dd s + \sqrt{2\epsilon} \tilde{W}^{\epsilon}(t).
\end{equation*}
Since the solutions to the equation above are unique in law (as a consequence of the Girsanov theorem, or by Proposition~\ref{prop:X} combined with the Yamada-Watanabe theorem), we deduce that the processes $\tilde{X}^{\epsilon}$ and $X^{\epsilon}$ have the same distribution. As a consequence, the process $Z^{\epsilon}$ has the same distribution as the process $\tilde{Z}^{\epsilon}$ defined by $\tilde{Z}^{\epsilon}(t) = \epsilon Z^1(t/\epsilon)$.

\subsubsection{Long time behaviour of $Z^1$} This paragraph is dedicated to the study of the stochastic differential equation
\begin{equation}\label{eq:ZMn}
  \forall t \geq 0, \qquad Z(t) = z^0 + \int_{s=0}^t b^{\Pi}(\Sigma Z(s))\dd s + \sqrt{2}\Pi W(t),
\end{equation}
where $z^0 \in M_n$. When $z^0=0$, the process $Z^1$ introduced above solves~\eqref{eq:ZMn}.

\begin{lem}\label{lem:ZMn}
  For all $z^0 \in M_n$, the stochastic differential equation~\eqref{eq:ZMn} admits a unique weak solution in $M_n$, defined on some probability space endowed with the probability distribution $P_{z^0}$ and the expectation $E_{z^0}$. It generates a Feller semigroup in $M_n$, in the sense that, for all continuous and bounded function $f : M_n \to \R$, the function $z \mapsto E_z(f(Z(t)))$ is continuous and bounded on $M_n$.
\end{lem}
\begin{proof}
  Any point $z=(z_1, \ldots, z_n) \in M_n$ is parametrized by the vector of its first $n-1$ coordinates $z'=(z_1, \ldots, z_{n-1}) \in \R^{n-1}$ through the continuous mapping $\varphi : z' \in \R^{n-1} \mapsto (z_1, \ldots, z_{n-1}, -(z_1 + \cdots + z_{n-1})) \in M_n$. Therefore, it is equivalent to prove weak existence and uniqueness and the Feller property for the stochastic differential equation
  \begin{equation}\label{eq:Zprime}
    \forall t \geq 0, \qquad Z'(t) = {z^0}' + \int_{s=0}^t (b^{\Pi})'(\Sigma \varphi(Z'(s)))\dd s + \sqrt{2}\Pi' W(t),
  \end{equation}
  in $\R^{n-1}$, where $(b^{\Pi})'(\sigma) = (b^{\Pi}_1(\sigma), \ldots, b^{\Pi}_{n-1}(\sigma))$ and $\Pi'$ is the rectangular matrix obtained by removing the $n$-th line from $\Pi$. Little algebra yields $\Pi'(\Pi')^* = I_{n-1} - (1/n)J_{n-1}$ which is positive definite. As a consequence, weak existence and uniqueness as well as the Feller property for~\eqref{eq:Zprime} follow from the Girsanov theorem.
\end{proof}

\begin{prop}\label{prop:pages}
  Under the strong stability condition~\eqref{eq:SSC}, the solution to~\eqref{eq:ZMn} admits a unique stationary probability distribution $\mu$, and it is {\em positive recurrent} in the sense that, for all measurable and bounded function $f : M_n \to \R$, 
  \begin{equation*}
    \forall z^0 \in M_n, \qquad \lim_{t \to +\infty} \frac{1}{t} \int_{s=0}^t f(Z(s)) \dd s = \int_{M_n} f \dd \mu, \qquad \text{$P_{z^0}$-almost surely}.
  \end{equation*}
\end{prop}
\begin{proof}
  The proof closely follows the lines of Pagès~\cite[Théorème~1, p.~148]{pages}, and we prove existence, uniqueness and positive recurrence separately.

  {\em Proof of existence.} The existence of a stationary probability distribution relies on the fact that the function $V$ defined on $M_n$ by $V(z)=\sum_{i=1}^n |z_i|^2$ is a Lyapunov function for~\eqref{eq:ZMn}. Indeed, let $L$ refer to the infinitesimal generator of $Z$. By the Itô formula,
  \begin{equation*}
    \forall z \in M_n, \qquad LV(z) = 2 \sum_{i=1}^n z_i b^{\Pi}_i(\Sigma z) + 2(n-1).
  \end{equation*}
  By Lemma~\ref{lem:lyap}, under the strong stability condition~\eqref{eq:SSC}, 
  \begin{equation*}
    LV(z) \leq -\bar{b}\max_{1 \leq i \leq n} |z_i| + 2(n-1),
  \end{equation*}
  and the conclusion follows from Ethier and Kurtz~\cite[Theorem~9.9, p.~243]{ethier}. 

  {\em Proof of uniqueness.} The uniqueness of a stationary probability distribution is a consequence of the regularity of the semigroup associated with the diffusion process $Z'$ in $\R^{n-1}$ introduced in the proof of Lemma~\ref{lem:ZMn}. More precisely, since $\Pi'(\Pi')^*$ is positive definite, it follows from the Girsanov theorem that, for all ${z^0}' \in \R^{n-1}$, for all $t > 0$, the distribution of $Z'(t)$ is equivalent to the Lebesgue measure on $\R^{n-1}$. By the same arguments as in the proof of~\cite[Proposition~8.1, p.~29]{rb}, this implies that the process $Z'$ does not admit more than one stationary probability distribution. The conclusion follows from the fact that the pushforward by the mapping $\varphi$ induces a one-to-one correspondance between the stationary distributions of $Z'$ and the stationary distributions of $Z$.

  {\em Proof of positive recurrence.} Since $\mu$ is the unique stationary probability distribution for the Feller process $Z$, it is ergodic~\cite[Proposition~3.5, p.~8]{rb}; therefore the pointwise ergodic theorem~\cite[Theorem~3.4, p.~8]{rb} ensures that, for all measurable and bounded function $f : M_n \to \R$,
  \begin{equation*}
    \text{for $\mu$-almost all $z^0 \in M_n$}, \quad \text{$P_{z^0}$-almost surely}, \qquad \lim_{t \to +\infty} \frac{1}{t} \int_{s=0}^t f(Z(s))\dd s = \int_{M_n} f \dd \mu.
  \end{equation*}
  The extension of this statement to {\em all} initial condition $z^0 \in M_n$ relies on the regularity of the semigroup associated with $Z$, and we refer to Pagès~\cite[Théorème~1, (b), p.~149]{pages} for a proof.
\end{proof}

\subsubsection{Velocity of the cluster} The description of the small noise limit can now be completed under the strong stability condition~\eqref{eq:SSC}.

\begin{prop}\label{prop:v}
  Under the strong stability condition~\eqref{eq:SSC}, the quantity 
  \begin{equation}\label{eq:v}
    v := \int_{z \in M_n} b_i(\Sigma z) \mu(\dd z), 
  \end{equation}
  with $\mu$ given by Proposition~\ref{prop:pages}, does not depend on $i \in \{1, \ldots, n\}$. Besides, $X^{\epsilon}$ converges in $L^2_{\loc}(\Pr_0)$ to $(vt, \ldots, vt)_{t \geq 0}$.
\end{prop}
\begin{proof}
  For all $\sigma \in S_n$, let $\zeta^{\epsilon}_{\sigma}$ refer to the occupation time of the process $\Sigma X^{\epsilon}$ in $\sigma$ defined by
  \begin{equation*}
    \forall t \geq 0, \qquad \zeta^{\epsilon}_{\sigma}(t) := \int_{s=0}^t \ind{\Sigma X^{\epsilon}(s) = \sigma} \dd s.
  \end{equation*}
  Certainly, for all $i \in \{1, \ldots, n\}$,
  \begin{equation*}
    \forall t \geq 0, \qquad X^{\epsilon}_i(t) = \sum_{\sigma \in S_n} b_i(\sigma) \zeta^{\epsilon}_{\sigma}(t) + \sqrt{2\epsilon} W_i(t).
  \end{equation*}

  On the other hand, for a fixed $t > 0$,
  \begin{equation*}
    \frac{\zeta^{\epsilon}_{\sigma}(t)}{t} = \frac{1}{t} \int_{s=0}^t \ind{\Sigma Z^{\epsilon}(s) = \sigma} \dd s
  \end{equation*}
  has the same distribution as
  \begin{equation*}
    \frac{1}{t} \int_{s=0}^t \ind{\Sigma \tilde{Z}^{\epsilon}(s) = \sigma} \dd s = \frac{1}{t} \int_{s=0}^t \ind{\Sigma (\epsilon Z^1(s/\epsilon))=\sigma} \dd s = \frac{\epsilon}{t} \int_{u=0}^{t/\epsilon} \ind{\Sigma  Z^1(u)=\sigma} \dd u.
  \end{equation*}
  By the weak uniqueness for the solution to~\eqref{eq:ZMn}, Proposition~\ref{prop:pages} can be applied to $Z^1$ and yields 
  \begin{equation*}
    \lim_{\epsilon \dto 0} \frac{\epsilon}{t} \int_{u=0}^{t/\epsilon} \ind{\Sigma  Z^1(u)=\sigma} \dd u = \int_{z \in M_n} \ind{\Sigma z = \sigma} \dd \mu, \qquad \text{$\Pr_0$-almost surely}.
  \end{equation*}
  
  Thus, for all $t \geq 0$, the random variable $R_i^{\epsilon}(t) := \sum_{\sigma \in S_n} b_i(\sigma) \zeta^{\epsilon}_{\sigma}(t)$ converges in probability, in $\R$, to the deterministic limit $v_i t$ where $v_i$ is the right-hand side of~\eqref{eq:v}. As a consequence, the process $R_i^{\epsilon}$ converges in finite-dimensional distribution to the process $v_i t$. On the other hand, since
  \begin{equation*}
    \forall 0 \leq s \leq t, \qquad |R_i^{\epsilon}(t)-R_i^{\epsilon}(s)| = \left| \int_{r=s}^t b_i(\Sigma X^{\epsilon}(r)) \dd r\right| \leq \max_{\sigma \in S_n} |b_i(\sigma)| (t-s),
  \end{equation*}
  the modulus of continuity of $R_i^{\epsilon}$ is uniformly bounded with respect to $\epsilon$. Therefore, by the Arzelà-Ascoli theorem, the family of the laws of $R_i^{\epsilon}$ is tight and, for all $T>0$, $R_i^{\epsilon}$ converges in probability, in $C([0,T],\R)$, to the deterministic process $v_i t$. Finally, since
  \begin{equation*}
    \forall t \in [0,T], \qquad |R^{\epsilon}_i(t)| \leq \max_{\sigma \in S_n} |b_i(\sigma)| T,
  \end{equation*}
  then $R_i^{\epsilon}$ is bounded on $[0,T]$ uniformly in $\epsilon$, therefore the convergence of $R_i^{\epsilon}$ to $v_i t$ also holds in $L^2_{\loc}(\Pr_0)$. As a consequence, $X^{\epsilon}_i = R^{\epsilon}_i + \sqrt{2\epsilon} W_i$ converges to $v_i t$ in $L^2_{\loc}(\Pr_0)$, so that $X^{\epsilon}$ converges to $(v_1 t, \ldots, v_n t)$ in $L^2_{\loc}(\Pr_0)$. The fact that $v_i$ does not depend on $i$ finally follows from Proposition~\ref{prop:singclus}.
\end{proof}

\begin{rk}\label{rk:velocity}
  In the two-particle case addressed in Section~\ref{s:n2}, the explicit computation of the velocity of the cluster as a function of $b$ was made possible by the fact that the two quantities $\zeta^{\epsilon}_{(12)}(t)$ and $\zeta^{\epsilon}_{(21)}(t)$ satisfy the two independent relations
  \begin{equation*}
    \begin{aligned}
      & \zeta^{\epsilon}_{(12)}(t) + \zeta^{\epsilon}_{(21)}(t) = t,\\
      & \lim_{\epsilon\dto 0} b_1(12) \zeta^{\epsilon}_{(12)}(t) + b_1(21) \zeta^{\epsilon}_{(21)}(t) = \lim_{\epsilon\dto 0} b_2(12) \zeta^{\epsilon}_{(12)}(t) + b_2(21) \zeta^{\epsilon}_{(21)}(t).
    \end{aligned}
  \end{equation*}
  
  As soon as $n \geq 3$, under the stability condition~\eqref{eq:SC}, the $n!$ unknown quantities $\zeta^{\epsilon}_{\sigma}(t)$, $\sigma \in S_n$ satisfy the $n$ independent relations 
  \begin{equation*}
    \begin{aligned}
      & \sum_{\sigma \in S_n} \zeta^{\epsilon}_{\sigma}(t) = t,\\
      & \text{for all $i \in \{1, \ldots, n\}$}, \quad \lim_{\epsilon \dto 0} \sum_{\sigma \in S_n} b_i(\sigma) \zeta^{\epsilon}_{\sigma}(t) \text{ does not depend on $i$},
    \end{aligned}
  \end{equation*}
  which is not enough to determine the small noise limit of the quantities $\zeta^{\epsilon}_{\sigma}(t)$, $\sigma \in S_n$.
  
  Under the strong stability condition~\eqref{eq:SSC}, another strategy to compute the velocity $v$ of the cluster consists in a straightfoward application of the formula~\eqref{eq:v}, which requires to compute $\mu$ by solving the elliptic problem $L^* \mu = 0$ on $M_n$, where the infinitesimal generator $L$ of the solution to~\eqref{eq:ZMn} is constant on each cone $\{z = (z_1, \ldots, z_n) \in M_n : \Sigma z = \sigma\}$, $\sigma \in S_n$. This task can be carried out in the rank-based case~\cite[Theorem~8, p.~2187]{pp}, and can easily be extended to perturbations of this case where, letting $b = (b_1, \ldots, b_n) \in \R^n$ as in Section~\ref{s:rank} and $b' : S_n \to \R$, the drift of the $i$-th particle in the configuration $\sigma$ is given by $b_{\sigma^{-1}(i)} + b'(\sigma)$. However, we were not able to extend this approach to the general order-based case.
\end{rk}

\subsection{A counterexample to necessariness}\label{ss:contrex} Unlike in the rank-based case, the stability condition~\eqref{eq:SC} and {\em a fortiori} the strong stability condition~\eqref{eq:SSC} are not necessary for all the particles to aggregate into a single cluster in the small noise limit. Indeed, consider the following example with $n=3$: let $b(123) = (\lambda_1, \lambda_2, \lambda_3)$, $b(132) = (\eta_1, \eta_3, \eta_2)$ and $b_{\sigma(1)}(\sigma) = 1$, $b_{\sigma(2)}(\sigma)=0$, $b_{\sigma(3)}(\sigma) = -1$ for all $\sigma \in S_3 \setminus \{(123),(132)\}$. We choose $(\lambda_1,\lambda_2,\lambda_3)$ and $(\eta_1,\eta_2,\eta_3)$ in such a way that the configuration $(123)$ does not satisfy the stability condition~\eqref{eq:SC}, which is the case if for instance $\lambda_1 < (\lambda_2+\lambda_3)/2$,
but the particles still aggregate into a single cluster in the small noise limit.

We only give the main idea of the counterexample, the details of the proof are of the same nature as in Appendix~\ref{app:n2}. When $X^{\epsilon}$ is not in the configurations $\{(123),(132)\}$, the instantaneous drifts of the particles tend to keep them close to each other. During an excursion of $X^{\epsilon}$ in the configurations $\{(123),(132)\}$, \ie an excursion of the first particle on the left of the two other particles, the average velocity of the first particle writes $v^1 = \rho \lambda_1 + (1-\rho)\eta_1$, where $\rho$ is the relative amount of time spent in the configuration $(123)$ during the excursion.

If the configurations $(123)$ and $(132)$ are such that $\lambda_2 > \lambda_3$ and $\eta_3 > \eta_2$, then in both configurations $(123)$ and $(132)$, the subsystem composed by the second and the third particles is converging/converging in the sense of Section~\ref{s:n2}. As a consequence, the relative amount of time $\rho$ spent in the configuration $(123)$ during the excursion approximately writes $\rho = (\eta_3-\eta_2)/(\lambda_2-\lambda_3 + \eta_3 - \eta_2)$. Therefore, during the excursion, the average velocity of the subsystem composed by the second and the third particles approximately writes 
\begin{equation*}
  v^{23} = \frac{\eta_3\lambda_2-\eta_2\lambda_3}{\lambda_2-\lambda_3 + \eta_3 - \eta_2}. 
\end{equation*}
Note that, by the definition of $\rho$,
\begin{equation*}
  v^{23} = \rho \lambda_2 + (1-\rho)\eta_2 = \rho \lambda_3 + (1-\rho)\eta_3 = \rho \frac{\lambda_2 + \lambda_3}{2} + (1-\rho)\frac{\eta_2 + \eta_3}{2}. 
\end{equation*}
The particles tend to get closer to each other if $v^1 \geq v^{23}$.

Let us fix some arbitrary values of $\lambda_2$, $\lambda_3$, $\eta_2$, $\eta_3$ such that $\lambda_2 > \lambda_3$ and $\eta_3 > \eta_2$. This prescribes a given value for $\rho \in (0,1)$. The key observation is that $\rho$ does not depend on the values of $\lambda_1$ and $\eta_1$. Of course, if $\lambda_1$ and $\eta_1$ are chosen so that both $(123)$ and $(132)$ satisfy the stability condition~\eqref{eq:SC}, then $\lambda_1 \geq (\lambda_2+\lambda_3)/2$ and $\eta_1 \geq (\eta_2+\eta_3)/2$, and the inequality $v^1 \geq v^{23}$ is straightforward. Let us now fix $\lambda_1 < (\lambda_2+\lambda_3)/2$, so that the configuration $(123)$ does not satisfy the stability condition~\eqref{eq:SC}: in this configuration, the first particle drifts away to the left of the second and the third particles. But, since $\rho$ and $v^{23}$ do not depend on the value of $\eta_1$, the latter can be taken large enough for the inequality $\rho \lambda_1 + (1-\rho)\eta_1 \geq v^{23}$ to hold, and therefore we recover $v^1 \geq v^{23}$. To sum up, the configuration $(132)$ can be chosen `converging enough' to balance the `diverging tendency' of the configuration $(123)$. As a consequence, the particles still aggregate into a single cluster in the small noise limit, while the stability condition~\eqref{eq:SC} is not satisfied. 

%%%%%%%%%%%%%%%%%%%%%%%%%%%%%%%%%%%%%%%%%%%%%%%%%%%%%%%%%%%%%%%%%%%%%%%%%%%%%%%%%%%%%%%%%%%%%%%%%%%%%%

\section{Conclusion}\label{s:conclusion}

Let us conclude this article by stating a few conjectures as regards the general behaviour of the process $X^{\epsilon}$ in the small noise limit. Excluding the degenerate situations such as the case $b^+=b^-=0$ in Section~\ref{s:n2} and recalling that, for all $\sigma \in S_n$, $\zeta^{\epsilon}_{\sigma}(t)$ is the occupation time of $\Sigma X^{\epsilon}$ in the configuration $\sigma$, we expect that the quantity
\begin{equation*}
  \rho_{\sigma} := \lim_{\epsilon \dto 0} \frac{1}{t} \zeta^{\epsilon}_{\sigma}(t)
\end{equation*}
does not depend on $t$ for $t < t^*$, where $t^*$ should be thought of as the smallest possible instant of collision between two particles with distinct initial position in the small noise limit. Note that $\rho=(\rho_{\sigma})_{\sigma \in S_n}$ is a probability distribution on $S_n$. It is either random, in which case the particle system in the small noise limit randomly selects a trajectory among several possible ones, or deterministic, in which case the motion of the particle system in the small noise limit is deterministic. For a given realization of $\rho$, the particles travel with constant velocity vector $b^{\rho} := \sum_{\sigma \in S_n} \rho_{\sigma} b(\sigma)$ on $[0,t^*]$. 

Let us fix a realization of $\rho$. Then either all the particles drift away from each other without aggregating into clusters, or several groups of particles aggregate into clusters. This is observed on $\rho$ as follows: in the first case, $\rho=\delta_{\sigma}$, where $\delta_{\sigma}$ is the Dirac measure in the configuration $\sigma$ corresponding to the order in which the particles drift away from each other. Then, $b^{\rho} = b(\sigma)$ and $b_{\sigma(1)}(\sigma) < \cdots < b_{\sigma(n)}(\sigma)$. In the second case, let $\{i_1, \ldots, j_1\}, \ldots, \{i_k, \ldots, j_k\}$ refer to the sets of indices composing each of the $k$ clusters, with $k \geq 1$, $i_1 < j_1 < \cdots < i_k < j_k$. Then, the support of $\rho$, \ie the set of $\sigma \in S_n$ such that $\rho_{\sigma} > 0$, is exactly described by the set of products $\sigma^1 \cdots \sigma^k$, where $(\sigma^1, \ldots, \sigma^k)$ is such that, for all $l \in \{1, \ldots, k\}$, $\sigma^l$ leaves the set $\{1, \ldots, n\} \setminus \{i_l, \ldots, j_l\}$ invariant. As is noted in Remark~\ref{rk:velocity}, the detailed computation of the weights $\rho_{\sigma}$ associated with such permutations remains an open question.

As far as the law of the random probability distribution $\rho$ is concerned, if there exists $\sigma \in S_n$ such that $b_{\sigma(1)}(\sigma) < \cdots < b_{\sigma(n)}(\sigma)$, then the support of the law of $\rho$ is given by the set of the Dirac distributions in each such $\sigma$. The weights associated with each such $\sigma$ can be computed by solving an elliptic problem similar to the one introduced in the proof of Lemma~\ref{lem:taudelta} in Appendix~\ref{app:n2}, in higher dimensions. To our knowledge, there is no explicit solution to such a multidimensional problem. 

If there is no permutation $\sigma \in S_n$ such that $b_{\sigma(1)}(\sigma) < \cdots < b_{\sigma(n)}(\sigma)$, then determining the law of $\rho$ in terms of $b$ amounts to determining the sets of particles that can form clusters with positive probability. This requires a combinatorial analysis of $b$ that remains unclear to us.

\sk
The analysis of collisions above allows us to provide a global description of the small noise limit of $X^{\epsilon}$: excluding again the degenerate situations such as the case $b^+=b^-=0$ in Section~\ref{s:n2}, then between two collisions, the particles travel with a constant velocity, either alone or into clusters, depending on the outcome of the latest collision. At each collision, the velocity of all the particles are modified, possibly randomly. The colliding particles can stick into clusters, and clusters of particles not involved in the collision can be splitted.

\sk
The small noise limit of $X^{\epsilon}$ somehow behaves like the {\em generalized flows} introduced by E and Vanden-Eijnden~\cite{eve}. Indeed, it follows a deterministic trajectory, that has to be interpreted as a solution to the zero noise ODE $\dot{x} = b(\Sigma x)$ in an appropriate sense, then randomly selects a new trajectory at each collision, \ie at each new singularity for the ODE. But whereas E and Vanden-Eijnden observed a loss of the Markov property for some particular examples of generalized flows, which was also the case in the work by Delarue, Flandoli and Vincenzi discussed in introduction~\cite{dfv}, we conjecture that in the order-based case, the small noise limit remains a (piecewise deterministic) Markov process. Indeed, the strong Markov property for the process $X^{\epsilon}$ induces a loss of memory at the collision (see the proof of Corollary~\ref{cor:Z} in Appendix~\ref{app:n2} below), so that the law of the small noise limit at a collision is the same as if the process restarts in the current position.

%%%%%%%%%%%%%%%%%%%%%%%%%%%%%%%%%%%%%%%%%%%%%%%%%%%%%%%%%%%%%%%%%%%%%%%%%%%%%%%%%%%%%%%%%%%%%%%%%%%%%%

\appendix

\section{Proofs in the two-particle case}\label{app:n2}

This appendix contains the remaining proofs in the two-particle case of Section~\ref{s:n2}; namely the proofs of cases~\eqref{case:didi:zeta}, \eqref{case:condi:zeta} and~\eqref{case:dicon:zeta} in Lemma~\ref{lem:zeta} and the proof of Corollary~\ref{cor:Z}. 

When the particles have the same initial position, cases~\eqref{case:didi:zeta}, \eqref{case:condi:zeta} and~\eqref{case:dicon:zeta} in Lemma~\ref{lem:zeta} correspond to situations in which the small noise limit of $X^{\epsilon}$ concentrates on the extremal solutions $x^-$ and $x^+$ associated with diverging configurations. Similarly to~\cite{bafico}, the computation of the weights associated with $x^-$ and $x^+$ in the diverging/diverging situation relies on the resolution of a one-dimensional elliptic problem. This task is carried out in Subsection~\ref{ss:didi}, in a slightly more general framework, independent of the remainder of this article. The proofs of cases~\eqref{case:didi:zeta}, \eqref{case:condi:zeta} and~\eqref{case:dicon:zeta} in Lemma~\ref{lem:zeta} are provided in Subsection~\ref{ss:pfs}.

The proof of Corollary~\ref{cor:Z}, which addresses the small noise limit of $Z^{\epsilon}$ when $z^0 \not= 0$, is given in Subsection~\ref{ss:corZ}.

\subsection{Auxiliary results in the diverging/diverging case}\label{ss:didi} Let $a^+ : [0,+\infty) \to \R$, $a^- : (-\infty,0] \to \R$ be bounded and continuous functions, such that $a^-(0) < 0$ and $a^+(0) > 0$. We define the function $a : \R \to \R$ by $a(z) := a^+(z)$ if $z > 0$, and $a(z) := a^-(z)$ if $z \leq 0$. By the Girsanov theorem, for all $z^0 \in \R$, the stochastic differential equation
\begin{equation}\label{eq:Zez}
  \forall t \geq 0, \qquad Z^{\epsilon}(t) = z^0 + \int_{s=0}^t a(Z^{\epsilon}(s)) \dd s + 2\sqrt{\epsilon} B(t),
\end{equation}
admits a unique weak solution defined on some probability space endowed with the probability distribution $P_{z^0}$. The expectation under $P_{z^0}$ is denoted by $E_{z^0}$.

\begin{lem}\label{lem:taudelta}
  Let $\bar{\delta} := 1 \wedge \inf\{\delta \geq 0 : a^+(\delta)=0 \text{ or } a^-(-\delta)=0\} > 0$. For all $\delta \in (0,\bar{\delta})$, let $\tau_{\delta} := \inf\{t \geq 0 : |Z^{\epsilon}(t)| = \delta\}$. Then, for all $z^0 \in [-\delta, \delta]$, $\tau_{\delta}$ is finite $P_{z^0}$-almost surely, and
  \begin{equation*}
    P_0(Z^{\epsilon}(\tau_{\delta})=\delta) = \left(1+\frac{\displaystyle\int_{y=0}^{\delta} \exp\left(-\frac{1}{2\epsilon}\int_{x=0}^y a^+(x)\dd x\right) \dd y}{\displaystyle\int_{y=-\delta}^0 \exp\left(\frac{1}{2\epsilon}\int_{x=y}^0 a^-(x)\dd x\right) \dd y}\right)^{-1}.
  \end{equation*}
\end{lem}

The limit of the quantity above when $\epsilon$ goes to $0$ is given by the following corollary. 

\begin{cor}\label{cor:taudelta}
  Under the assumptions of Lemma~\ref{lem:taudelta}, for any $\bar{\epsilon} > 0$ and any function $\delta : (0,\bar{\epsilon}) \to (0, \bar{\delta})$ such that $\epsilon/\delta(\epsilon)$ vanishes with $\epsilon$, then
  \begin{equation*}
    \lim_{\epsilon \dto 0} P_0(Z^{\epsilon}(\tau_{\delta(\epsilon)})=\delta(\epsilon)) = \frac{a^+(0)}{a^+(0)-a^-(0)}.
  \end{equation*}
\end{cor}

\begin{proof}[Proof of Lemma~\ref{lem:taudelta}]
  Let $z^0 \in \R$. Under $P_{z^0}$, for all $t \geq 0$, the Itô-Tanaka formula writes
  \begin{equation*}
    |Z^{\epsilon}(t)| = |z^0| +\int_{s=0}^t \sgn(Z^{\epsilon}(s))a(Z^{\epsilon}(s))\dd s + 2\sqrt{\epsilon} \tilde{B}(t) + L^{\epsilon}(t),
  \end{equation*}
  where the local time $L^{\epsilon}$ at $0$ of the semimartingale $Z^{\epsilon}$ is a nonnegative process, and the process $\tilde{B}$ defined by
  \begin{equation*}
    \forall t \geq 0, \qquad \tilde{B}(t) := \int_{s=0}^t \sgn(Z^{\epsilon}(s))\dd B(s)
  \end{equation*}
  is a Brownian motion, due to Lévy's characterization. Since $\delta < \bar{\delta}$, for all $t \leq \tau_{\delta}$, $a^+(|Z^{\epsilon}(t)|) > 0$ and $a^-(-|Z^{\epsilon}(t)|) < 0$, so that $\sgn(Z^{\epsilon}(t))a(Z^{\epsilon}(t)) \geq 0$ and $|Z^{\epsilon}(t)| \geq |z^0| + 2\sqrt{\epsilon} \tilde{B}(t)$. As a consequence, if $|z^0| \leq \delta$ then $\tau_{\delta} \leq \inf\{t>0 : |z^0|+2\sqrt{\epsilon}\tilde{B}(t) = \delta\}$, which is known to be finite $P_{z^0}$-almost surely~\cite[Remark~8.3, p.~96]{karatzas}. Hence, $\tau_{\delta}$ is finite almost surely.
  
  Let $u$ be the solution to the elliptic problem on $[-\delta,\delta]$:
  \begin{equation*}
    \left\{\begin{aligned}
      & 2\epsilon u''(z) + a(z) u'(z)=0,\\
      & u(-\delta)=0, \quad u(\delta)=1,\\
    \end{aligned}\right.
  \end{equation*}
  given by
  \begin{equation*}
    \forall z \in [-\delta, \delta], \qquad u(z) = \frac{\displaystyle\int_{y=-\delta}^z \exp(-A(y)/2\epsilon) \dd y}{\displaystyle\int_{y=-\delta}^{\delta} \exp(-A(y)/2\epsilon) \dd y},
  \end{equation*}
  where $A(y) := \int_{x=0}^z a(x)\dd x$. Then $u$ is $C^1$ on $[-\delta,\delta]$, and $u'$ is absolutely continuous with respect to the Lebesgue measure, so that, under $P_{z^0}$, $u(Z^{\epsilon}(\cdot \wedge \tau_{\delta}))$ is a martingale. By the martingale stopping theorem, for all $t \geq 0$, $E_{z^0}(u(Z^{\epsilon}(t \wedge \tau_{\delta}))) = u(z^0)$ and the dominated convergence theorem now yields $E_{z^0}(u(Z^{\epsilon}(\tau_{\delta}))) = P_{z^0}(Z^{\epsilon}(\tau_{\delta})=\delta) = u(z^0)$. The conclusion follows from taking $z^0=0$.
\end{proof}

\begin{proof}[Proof of Corollary~\ref{cor:taudelta}]
  The proof is based on the Laplace method. More precisely, we prove that
  \begin{equation*}
    \int_{y=0}^{\delta(\epsilon)} \exp\left(-\frac{1}{2\epsilon}\int_{x=0}^y a^+(x)\dd x\right) \dd y \mathop{\sim}\limits_{\epsilon \dto 0} \frac{2\epsilon}{a^+(0)},
  \end{equation*}
  and the same arguments lead to
  \begin{equation*}
    \int_{y=-\delta(\epsilon)}^0 \exp\left(\frac{1}{2\epsilon}\int_{x=y}^0 a^-(x)\dd x\right) \dd y \mathop{\sim}\limits_{\epsilon \dto 0} -\frac{2\epsilon}{a^-(0)},
  \end{equation*}
  which yields the expected result. Let us fix $\eta \in (0,1)$. Then by the right continuity of $a^+$ in $0$, there exists $x_0>0$ such that, for all $x \in [0,x_0]$, $(1-\eta) a^+(0) \leq a^+(x) \leq (1+\eta) a^+(0)$. As a consequence, 
  \begin{equation*}
    \begin{aligned}
      \int_{y=0}^{\delta(\epsilon) \wedge x_0} \exp\left(-\frac{(1+\eta)a^+(0)}{2\epsilon}y\right)\dd y & \leq \int_{y=0}^{\delta(\epsilon) \wedge x_0} \exp\left(-\frac{1}{2\epsilon}\int_{x=0}^y a^+(x)\dd x\right) \dd y\\
      & \leq \int_{y=0}^{\delta(\epsilon) \wedge x_0} \exp\left(-\frac{(1-\eta)a^+(0)}{2\epsilon}y\right)\dd y .
    \end{aligned}
  \end{equation*}
  Computing both the left- and the right-hand side above and using the facts that $a^+(0)>0$ and $\delta(\epsilon)/\epsilon$ goes to $+\infty$ when $\epsilon$ goes to $0$, we deduce that
  \begin{equation*}
    \begin{aligned}
      & \liminf_{\epsilon \dto 0} \frac{a^+(0)}{2\epsilon} \int_{y=0}^{\delta(\epsilon) \wedge x_0} \exp\left(-\frac{1}{2\epsilon}\int_{x=0}^y a^+(x)\dd x\right) \dd y \geq \frac{1}{1+\eta},\\
      & \limsup_{\epsilon \dto 0} \frac{a^+(0)}{2\epsilon} \int_{y=0}^{\delta(\epsilon) \wedge x_0} \exp\left(-\frac{1}{2\epsilon}\int_{x=0}^y a^+(x)\dd x\right) \dd y \leq \frac{1}{1-\eta}.
    \end{aligned}
  \end{equation*}
  Furthermore,
  \begin{equation*}
    \begin{aligned}
      & \frac{a^+(0)}{2\epsilon} \int_{y=\delta(\epsilon) \wedge x_0}^{\delta(\epsilon)} \exp\left(-\frac{1}{2\epsilon}\int_{x=0}^y a^+(x)\dd x\right) \dd y\\
      & \qquad \leq \ind{\delta(\epsilon)>x_0}\frac{a^+(0)}{2\epsilon} \int_{y=x_0}^{\delta(\epsilon)} \exp\left(-\frac{1}{2\epsilon}\int_{x=0}^{x_0} a^+(x)\dd x\right) \dd y\\
      & \qquad \leq \ind{\delta(\epsilon)>x_0}\frac{a^+(0)}{2\epsilon}\exp\left(-\frac{1}{2\epsilon}\int_{x=0}^{x_0} a^+(x)\dd x\right),
    \end{aligned}
  \end{equation*}
  where we used the fact that $\delta(\epsilon) \leq \bar{\delta} \leq 1$ by definition. The right-hand side above certainly vanishes when $\epsilon$ goes to $0$. Since $\eta$ is arbitrary, the proof is completed.
\end{proof}

\subsection{Remaining proofs in Lemma~\ref{lem:zeta}}\label{ss:pfs} Since, for all $t \geq 0$,
\begin{equation*}
  \zeta^{\epsilon}(t) = \int_{s=0}^t \ind{Z^{\epsilon}(s) \leq 0} \dd s,
\end{equation*}
the process $\zeta^{\epsilon}$ is measurable with respect to the filtration generated by the Brownian motion $B = (W_1-W_2)/\sqrt{2}$. Therefore, the convergences of cases~\eqref{case:didi:zeta}, \eqref{case:condi:zeta} and~\eqref{case:dicon:zeta} Lemma~\ref{lem:zeta} are stated in $L^1_{\loc}(\Pr_0)$, where the index $0$ stands for the value of $z^0$.

Let us begin with the proof of case~\eqref{case:condi:zeta}. Since case~\eqref{case:dicon:zeta} is symmetric, the proof is the same.

\begin{proof}[Proof of~\eqref{case:condi:zeta}]
  Let us assume that $b^+ > 0$, $b^- \geq 0$ and fix $T>0$. For all $t \in [0,T]$,
  \begin{equation*}
    \Exp_0\left(\sup_{t \in [0,T]} \zeta^{\epsilon}(t)\right) \leq \int_{s=0}^T \Pr_0(Z^{\epsilon}(s) \leq 0) \dd s.
  \end{equation*}
  
  Before proving that, for all $s \in [0,T]$, $\Pr_0(Z^{\epsilon}(s) \leq 0)$ vanishes with $\epsilon$ and concluding thanks to the dominated convergence theorem, let us make the two following remarks.
  \begin{itemize}
    \item Certainly, for all $s \geq 0$, $Z^{\epsilon}(s) \geq (b^+ \wedge b^-) t + 2\sqrt{\epsilon} B(s)$. Then, as soon as $b^- > 0$,
    \begin{equation*}
      \forall s \in [0,T], \qquad \Pr_0(Z^{\epsilon}(s) \leq 0) \leq \Pr_0\left(B(s) \leq -\frac{(b^+ \wedge b^-)s}{2\sqrt{\epsilon}}\right)
    \end{equation*}
    and the right-hand side vanishes with $\epsilon$.
    \item In the general case, the density of $Z^{\epsilon}(s)$ was derived by Karatzas and Shreve~\cite{ks} but its integration over the half line $(-\infty,0]$ is not an easy computation.
  \end{itemize}
  We provide a rather elementary proof, based on the use of hitting times of the Brownian motion and the strong Markov property for $Z^{\epsilon}$~\cite[Theorem~6.2.2, p.~146]{stroock}. For all $\delta > 0$, let us define $\tau_{\delta} := \inf\{t > 0 : Z^{\epsilon} = \delta\}$. Then, for all $s \in [0,T]$,
  \begin{equation}\label{eq:pf:condi:1}
    \begin{aligned}
      \Pr_0(Z^{\epsilon}(s) \leq 0) & = \Pr_0(Z^{\epsilon}(s) \leq 0, \tau_{\delta} > s) + \Pr_0(Z^{\epsilon}(s) \leq 0, \tau_{\delta} \leq s)\\
      & \leq \Pr_0(\tau_{\delta} > s) + \int_{t=0}^s \Pr_0(Z^{\epsilon}(s) \leq 0, \tau_{\delta} \in \dd t).
    \end{aligned}
  \end{equation}
  
  In the sequel, we shall choose $\delta$ as a function of $\epsilon$, going to $0$ with $\epsilon$, at a rate ensuring that both terms in the right-hand side above vanish. 
  
  Let us address the first of these terms. For all $t \geq 0$, $Z^{\epsilon}(t) \geq 2\sqrt{\epsilon}B(t)$, therefore $\tau_{\delta} \leq \sigma_{\delta} := \inf\{t > 0 : 2\sqrt{\epsilon} B(t) = \delta\}$. Following~\cite[Remark~8.3, p.~96]{karatzas}, $\sigma_{\delta}$ converges in probability to $0$ as soon as $\delta/\sqrt{\epsilon}$ goes to $0$. Under this condition, $\Pr_0(\tau_{\delta} > s)$ vanishes for all $s>0$.
  
  Let us now address the second term in the right-hand side of~\eqref{eq:pf:condi:1}. By the strong Markov property,
  \begin{equation*}
    \begin{aligned}
      & \int_{t=0}^s \Pr_0(Z^{\epsilon}(s) \leq 0, \tau_{\delta} \in \dd t) = \int_{t=0}^s \Pr_0(Z^{\epsilon}(s) \leq 0 | \tau_{\delta} = t)\Pr_0(\tau_{\delta} \in \dd t)\\
      & \leq \int_{t=0}^s \Pr_0(\inf\{r \geq t : Z^{\epsilon}(r) = 0\} < +\infty | \tau_{\delta} = t)\Pr_0(\tau_{\delta} \in \dd t)\\
      & = \int_{t=0}^s \Pr_0(\inf\{r \geq t : \delta + b^+ (r-t) + 2\sqrt{\epsilon} (B(r)-B(t)) = 0\} < +\infty | \tau_{\delta} = t)\Pr_0(\tau_{\delta} \in \dd t)\\
      & = \int_{t=0}^s \Pr_0(\inf\{r \geq 0 : \delta + b^+ r + 2\sqrt{\epsilon} B(r) = 0\} < +\infty)\Pr_0(\tau_{\delta} \in \dd t)\\
      & \leq \Pr_0(\inf\{r \geq 0 : \delta + b^+ r + 2\sqrt{\epsilon} B(r) = 0\} < +\infty).
    \end{aligned}
  \end{equation*}
  By~\cite[pp.~196-197]{karatzas}, $\Pr_0(\inf\{r \geq 0 : \delta + b^+ r + 2\sqrt{\epsilon} B(r) = 0\} < +\infty) = \exp(-b^+\delta/2\epsilon)$, and the latter vanishes as soon as $\epsilon/\delta$ goes to $0$. As a conclusion, taking $\delta = \epsilon^{3/4}$ allows to prove that the right-hand side of~\eqref{eq:pf:condi:1} vanishes with $\epsilon$, and the proof is completed.  
\end{proof}

We now address case~\eqref{case:didi:zeta}.

\begin{proof}[Proof of case~\eqref{case:didi:zeta}]
  Let us assume that $b^+>0$, $b^-<0$ and fix $T>0$. Let $F : C([0,T],\R) \to \R$ be bounded and Lipschitz continuous, with unit Lipschitz norm. Our purpose is to prove that
  \begin{equation*}
    \lim_{\epsilon \dto 0} \Exp_0(F(\zeta^{\epsilon})) = \frac{b^+}{b^+-b^-} F(0) + \frac{-b^-}{b^+-b^-} F(t),
  \end{equation*}
  where we recall that $t$ denotes the process $(t)_{t \geq 0}$. Then the conclusion follows from the Portmanteau theorem~\cite[Theorem~2.1, p.~11]{billingsley}.
  
  For $\delta>0$, let $\tau_{\delta} := \inf\{t > 0 : |Z^{\epsilon}(t)| = \delta\}$. Note that the definition of $\tau_{\delta}$ is not the same as in the proof of case~\eqref{case:condi:zeta} because of the absolute value. Then
  \begin{equation*}
    \begin{aligned}
      & \left|\Exp_0(F(\zeta^{\epsilon})) - \frac{b^+}{b^+-b^-} F(0) - \frac{-b^-}{b^+-b^-} F(t)\right|\\
      & \qquad \leq \left|\Exp_0\left(F(\zeta^{\epsilon})\ind{Z^{\epsilon}(\tau_{\delta}) = \delta}\right) - \frac{b^+}{b^+-b^-} F(0)\right| + \left|\Exp_0\left(F(\zeta^{\epsilon})\ind{Z^{\epsilon}(\tau_{\delta}) = -\delta}\right) - \frac{-b^-}{b^+-b^-} F(t)\right|,
    \end{aligned}
  \end{equation*}
  and we prove that the first term of the right-hand side above vanishes with $\epsilon$. The same arguments work for the second term.
  
  By the Lipschitz continuity of $F$,
  \begin{equation*}
    \begin{aligned}
      & \left|\Exp_0\left(F(\zeta^{\epsilon})\ind{Z^{\epsilon}(\tau_{\delta}) = \delta}\right) - \frac{b^+}{b^+-b^-} F(0)\right|\\
      & \qquad \leq \left|\Exp_0\left((F(\zeta^{\epsilon})-F(0))\ind{Z^{\epsilon}(\tau_{\delta}) = \delta}\right)\right| + \left|F(0)\left(\Pr_0(Z^{\epsilon}(\tau_{\delta}) = \delta) - \frac{b^+}{b^+-b^-}\right)\right|\\
      & \qquad \leq \Exp_0\left(\ind{Z^{\epsilon}(\tau_{\delta}) = \delta} \sup_{t \in [0,T]} \zeta^{\epsilon}(t)\right) + \left|F(0)\left(\Pr_0(Z^{\epsilon}(\tau_{\delta}) = \delta) - \frac{b^+}{b^+-b^-}\right)\right|.
    \end{aligned}
  \end{equation*}
  
  Owing to the uniqueness in law of solutions to~\eqref{eq:Zez} above, Corollary~\ref{cor:taudelta} ensures that the second term in the right-hand side above vanishes as soon as $\epsilon/\delta$ goes to $0$. The first term satisfies
  \begin{equation*}
    \begin{aligned}
      \Exp_0\left(\ind{Z^{\epsilon}(\tau_{\delta}) = \delta} \sup_{t \in [0,T]} \zeta^{\epsilon}(t)\right) & = \Exp_0\left(\ind{Z^{\epsilon}(\tau_{\delta}) = \delta} \int_{s=0}^T \ind{Z^{\epsilon}(s) \leq 0} \dd s\right)\\
      & = \int_{s=0}^T \Pr_0(Z^{\epsilon}(\tau_{\delta}) = \delta, Z^{\epsilon}(s) \leq 0) \dd s.
    \end{aligned}
  \end{equation*}
  We now prove that, for all $s \in [0,T]$, $\Pr_0(Z^{\epsilon}(\tau_{\delta}) = \delta, Z^{\epsilon}(s) \leq 0)$ vanishes for a suitable choice of $\delta$ depending on $\epsilon$. By the same arguments as in the proof of~\eqref{case:condi:zeta},
  \begin{equation*}
    \begin{aligned}
      & \Pr_0(Z^{\epsilon}(\tau_{\delta}) = \delta, Z^{\epsilon}(s) \leq 0)\\
      & \qquad = \Pr_0(Z^{\epsilon}(\tau_{\delta}) = \delta, Z^{\epsilon}(s) \leq 0, \tau_{\delta} > s) + \Pr_0(Z^{\epsilon}(\tau_{\delta}) = \delta, Z^{\epsilon}(s) \leq 0, \tau_{\delta} \leq s)\\
      & \qquad \leq \Pr_0(\tau_{\delta} > s) + \Pr_0(\inf\{r \geq 0 : \delta + b^+ r + 2\sqrt{\epsilon} B(r) = 0\} < +\infty)\\
      & \qquad = \Pr_0(\tau_{\delta} > s) + \exp(-b^+\delta/2\epsilon).
    \end{aligned} 
  \end{equation*}
  The second term in the right-hand side above vanishes as soon as $\epsilon/\delta$ goes to $0$. To control the first term, let us use the Itô-Tanaka formula and compute
  \begin{equation*}
    |Z^{\epsilon}(t)| = \int_{s=0}^t \sgn(Z^{\epsilon}(s))\ell(Z^{\epsilon}(s))\dd s + 2\sqrt{\epsilon}\int_{s=0}^t \sgn(Z^{\epsilon}(s)) \dd B(s) + L^{\epsilon}(t),
  \end{equation*}
  where the local time $L^{\epsilon}$ at $0$ of the semimartingale $Z^{\epsilon}$ is a nonnegative process. Besides, for all $z \in \R$, $\sgn(z)\ell(z) \geq 0$ and the process $\tilde{B}$ defined by
  \begin{equation*}
    \tilde{B}(t) = \int_{s=0}^t \sgn(Z^{\epsilon}(s)) \dd B(s)
  \end{equation*}
  is a Brownian motion, due to Lévy's characterization. As a consequence, $|Z^{\epsilon}(t)| \geq 2\sqrt{\epsilon}\tilde{B}(t)$, therefore $\tau_{\delta} \leq \sigma_{\delta} := \inf\{t \geq 0 : 2\sqrt{\epsilon} \tilde{B}(t) = \delta\}$. By the same argument as in the proof of~\eqref{case:condi:zeta}, $\Pr_0(\tau_{\delta} > s)$ vanishes as soon as $\delta/\sqrt{\epsilon}$ goes to $0$. We complete the proof by letting $\delta = \epsilon^{3/4}$.
\end{proof}

\subsection{Proof of Corollary~\ref{cor:Z}}\label{ss:corZ} Certainly, the cases $z^0>0$ and $z^0<0$ are symmetric, therefore we only address the case $z^0>0$. Recall that, in this case, the process $z^{\dto}$ is defined by:
\begin{itemize}
  \item if $b^+ \geq 0$, $z^{\dto}(t) = z^0 + b^+ t$ for all $t \geq 0$;
  \item if $b^+ < 0$ and $b^- \geq 0$, $z^{\dto}(t) = z^0 + b^+t$ if $t < t^* := z^0/(-b^+)$ and $z^{\dto}(t) = 0$ for $t \geq t^*$;
  \item if $b^+ < 0$ and $b^- < 0$, $z^{\dto}(t) = z^0 + b^+t$ if $t < t^*$ and $z^{\dto}(t) = b^-(t-t^*)$ for $t \geq t^*$.
\end{itemize}

\begin{proof}[Proof of Corollary~\ref{cor:Z}]
  Let us assume that $z^0>0$. Let $\tau_{\epsilon} := \inf\{t \geq 0 : Z^{\epsilon}(t)=0\} = \inf\{t \geq 0 : z^0 + b^+t + 2 \sqrt{\epsilon} B(t) = 0\}$. Following Karatzas and Shreve~\cite[Exercise~5.10, p.~197]{karatzas}, the Laplace transform of $\tau_{\epsilon}$ writes
  \begin{equation*}
    \forall \alpha > 0, \qquad \Exp_{z^0}(\exp\left(-\alpha \tau_{\epsilon}\right)) = \exp\left(-\frac{b^+ z^0}{4 \epsilon} - \frac{z^0}{2\sqrt{\epsilon}}\sqrt{\frac{(b^+)^2}{4\epsilon} + 2 \alpha}\right),
  \end{equation*}
  so that one easily deduces that:
  \begin{itemize}
    \item if $b^+ \geq 0$, then for all $T>0$, $\lim_{\epsilon \dto 0} \Pr_{z^0}(\tau_{\epsilon} \leq T) = 0$,
    \item if $b^+ < 0$, then $\tau_{\epsilon}$ converges in probability to $t^* = z^0/(-b^+)$.
  \end{itemize}
  
  We first address the case $b^+ \geq 0$. Then, for all $T > 0$,
  \begin{equation*}
    \begin{aligned}
      & \Exp_{z^0}\left(\sup_{t \in [0,T]} |Z^{\epsilon}(t) - z^{\dto}(t)|\right)\\
      & \qquad = \Exp_{z^0}\left(\sup_{t \in [0,T]} |Z^{\epsilon}(t) - z^{\dto}(t)|\ind{\tau_{\epsilon} \leq T}\right) + \Exp_{z^0}\left(\sup_{t \in [0,T]} |Z^{\epsilon}(t) - z^{\dto}(t)|\ind{\tau_{\epsilon} > T}\right),
    \end{aligned}
  \end{equation*}
  and
  \begin{equation*}
    \begin{aligned}
      & \Exp_{z^0}\left(\sup_{t \in [0,T]} |Z^{\epsilon}(t) - z^{\dto}(t)|\ind{\tau_{\epsilon} \leq T}\right) = \Exp_{z^0}\left(\sup_{t \in [0,T]} |(b^--b^+)\zeta^{\epsilon}(t) + 2\sqrt{\epsilon}B(t)|\ind{\tau_{\epsilon} \leq T}\right)\\
      & \qquad \leq |b^--b^+|T \Pr(\tau_{\epsilon} \leq T) + 2\sqrt{\epsilon}\Exp_{z^0}\left(\sup_{t \in [0,T]} |B(t)|\ind{\tau_{\epsilon} \leq T}\right),
    \end{aligned}
  \end{equation*}
  while
  \begin{equation*}
    \Exp_{z^0}\left(\sup_{t \in [0,T]} |Z^{\epsilon}(t) - z^{\dto}(t)|\ind{\tau_{\epsilon} > T}\right) = 2\sqrt{\epsilon}\Exp_{z^0}\left(\sup_{t \in [0,T]} |B(t)|\ind{\tau_{\epsilon} > T}\right).
  \end{equation*}
  As a consequence, 
  \begin{equation*}
    \Exp_{z^0}\left(\sup_{t \in [0,T]} |Z^{\epsilon}(t) - z^{\dto}(t)|\right) \leq |b^--b^+|T \Pr(\tau_{\epsilon} \leq T) + 2\sqrt{\epsilon}\Exp_{z^0}\left(\sup_{t \in [0,T]} |B(t)|\right),
  \end{equation*}
  and the right-hand side above easily vanishes with $\epsilon$.
  
  We now address the case $b^+ < 0$. Let us first define the random process $z^{\dto}_{\epsilon}$ by
  \begin{equation*}
    \forall t \geq 0, \qquad z^{\dto}_{\epsilon}(t) := \left\{\begin{aligned}
      & z^0 + b^+t & \text{if $t < \tau_{\epsilon}$},\\
      & 0 & \text{if $t \geq \tau_{\epsilon}$ and $b^- \geq 0$},\\
      & b^-(t-\tau_{\epsilon}) & \text{if $t \geq \tau_{\epsilon}$ and $b^- < 0$}.
    \end{aligned}\right.    
  \end{equation*}
  Note that, for $t \geq \tau_{\epsilon}$, $z^{\dto}_{\epsilon}(t)$ writes $b'(t-\tau_{\epsilon})$, where $b' := b^- \wedge 0$. We now prove that, for all $T > 0$,
  \begin{equation*}
    \lim_{\epsilon \dto 0} \Exp_{z^0}\left(\sup_{t \in [0,T]} |Z^{\epsilon}(t) - z^{\dto}_{\epsilon}(t)|\right) = 0.
  \end{equation*}
  In this purpose, we fix $T>0$ and write, on the one hand,
  \begin{equation*}
    \Exp_{z^0}\left(\sup_{t \in [0,\tau_{\epsilon} \wedge T)} |Z^{\epsilon}(t) - z^{\dto}_{\epsilon}(t)|\right) = \Exp_{z^0}\left(\sup_{t \in [0,\tau_{\epsilon} \wedge T)} |2\sqrt{\epsilon} B(t)|\right) \leq 2\sqrt{\epsilon} \Exp_{z^0}\left(\sup_{t \in [0,T]} |B(t)|\right).
  \end{equation*}
  On the other hand,
  \begin{equation*}
    \begin{aligned}
      \Exp_{z^0}\left(\sup_{t \in [\tau_{\epsilon} \wedge T, T]} |Z^{\epsilon}(t) - z^{\dto}_{\epsilon}(t)|\right) & = \Exp_{z^0}\left(\ind{\tau_{\epsilon} \leq T} \sup_{t \in [\tau_{\epsilon}, T]} |Z^{\epsilon}(t) - b'(t-\tau_{\epsilon})|\right)\\
      & \leq \Exp_{z^0}\left(\sup_{s \in [0,T]} |Z^{\epsilon}(s+\tau_{\epsilon}) - b's|\right)\\
      & = \Exp_{z^0}\left(\Exp_{z^0}\left(\left.\sup_{s \in [0,T]} |Z^{\epsilon}(s+\tau_{\epsilon}) - b's| \right| \mathcal{F}_{\tau_{\epsilon}}\right)\right)\\
      & = \Exp_0\left(\sup_{s \in [0,T]} |Z^{\epsilon}(s) - b's|\right),
    \end{aligned}
  \end{equation*}
  where we have used the strong Markov property for the process $(Z^{\epsilon}(t))_{t \geq 0}$ and the fact that $Z^{\epsilon}(\tau_{\epsilon}) = 0$. It now follows from Proposition~\ref{prop:Z} that the right-hand side above vanishes with $\epsilon$.
  
  To complete the proof, we finally check that
  \begin{equation}\label{eq:limzdtozdtoeps}
    \lim_{\epsilon \dto 0} \Exp_{z^0}\left(\sup_{t \in [0,T]} |z^{\dto}_{\epsilon}(t) - z^{\dto}(t)|\right) = 0.
  \end{equation}
  It follows from a straightforward analysis of $|z^{\dto}_{\epsilon}(t) - z^{\dto}(t)|$ that there exists $C > 0$, depending on $b^+$ and $b'$, such that, for all $t \in [0,T]$, $|z^{\dto}_{\epsilon}(t) - z^{\dto}(t)| \leq C (T \wedge |\tau_{\epsilon}-t^*|)$. Since $\tau_{\epsilon}$ converges in probability to $t^*$ and the function $t \mapsto C (T \wedge |t-t^*|)$ is continuous and bounded, we obtain~\eqref{eq:limzdtozdtoeps} and the proof is completed.
\end{proof}
%%%%%%%%%%%%%%%%%%%%%%%%%%%%%%%%%%%%%%%%%%%%%%%%%%%%%%%%%%%%%%%%%%%%%%%%%%%%%%%%%%%%%%%%%%%%%%%%%%%%%%

\subsection*{Acknowledgements} We are grateful to Régis Monneau for stimulating discussions that motivated this study. This work also benefited from fruitful conversations with Franco Flandoli and Lorenzo Zambotti. Finally, we would like to thank the anonymous referees for their careful reading of the manuscript. Their suggestions allowed us to improve the presentation of several proofs.

\end{document}